\documentclass{article}
\usepackage{amsmath,amsthm,amssymb}
\usepackage[usenames,dvipsnames]{xcolor}
\usepackage{enumerate}
\usepackage{graphicx}
\usepackage{cite}
\usepackage{comment}
\usepackage{oands}
\usepackage{tikz}
\usepackage{changepage}
\usepackage{bbm}
\usepackage{mathtools}
\usepackage[margin=1in]{geometry}
\usepackage[pagewise,mathlines]{lineno}
\usepackage{appendix}
\usepackage{stmaryrd} 
\usepackage{multicol}
\usepackage{microtype}

\usepackage[pdftitle={Chordal SLE$_6$ explorations of a quantum disk},
  pdfauthor={Ewain Gwynne and Jason Miller},
colorlinks=true,linkcolor=NavyBlue,urlcolor=RoyalBlue,citecolor=PineGreen,bookmarks=true,bookmarksopen=true,bookmarksopenlevel=2,unicode=true,linktocpage]{hyperref}

\theoremstyle{plain}
\newtheorem{thm}{Theorem}[section]

\newtheorem{lem}[thm]{Lemma}

\def\@rst #1 #2other{#1}
\newcommand\MR[1]{\relax\ifhmode\unskip\spacefactor3000 \space\fi
  \MRhref{\expandafter\@rst #1 other}{#1}}
\newcommand{\MRhref}[2]{\href{http://www.ams.org/mathscinet-getitem?mr=#1}{MR#2}}

\theoremstyle{definition}
\newtheorem{defn}[thm]{Definition}

\numberwithin{equation}{section}

\newcommand{\dsb}{\begin{adjustwidth}{2.5em}{0pt}
\begin{footnotesize}}
\newcommand{\dse}{\end{footnotesize}
\end{adjustwidth}}

\newcommand{\ssb}{\begin{adjustwidth}{2.5em}{0pt}}
\newcommand{\sse}{\end{adjustwidth}}

\newcommand{\aryb}{\begin{eqnarray*}}
\newcommand{\arye}{\end{eqnarray*}}
\def\alb#1\ale{\begin{align*}#1\end{align*}}
\def\allb#1\alle{\begin{align}#1\end{align}}
\newcommand{\eqb}{\begin{equation}}
\newcommand{\eqe}{\end{equation}}
\newcommand{\eqbn}{\begin{equation*}}
\newcommand{\eqen}{\end{equation*}}

\newcommand{\BB}{\mathbbm}
\newcommand{\ol}{\overline}
\newcommand{\ul}{\underline}
\newcommand{\op}{\operatorname}

\newcommand{\frk}{\mathfrak}
\newcommand{\eqD}{\overset{d}{=}}
\newcommand{\ep}{\epsilon}
\newcommand{\rta}{\rightarrow}

\newcommand{\wt}{\widetilde}
\newcommand{\wh}{\widehat} 
\newcommand{\mcl}{\mathcal}

\newcommand{\bdy}{\partial}

\newcommand{\bead}{{\operatorname{b}}}
\newcommand{\wge}{{\infty}}
\newcommand{\sfi}{{\operatorname{a}}} 
\newcommand{\nat}{{}}

\let\originalleft\left
\let\originalright\right
\renewcommand{\left}{\mathopen{}\mathclose\bgroup\originalleft}
\renewcommand{\right}{\aftergroup\egroup\originalright}

\title{Chordal SLE$_6$ explorations of a quantum disk}
\date{  }
\author{
\begin{tabular}{c} Ewain Gwynne\\[-5pt]\small MIT \end{tabular}
\begin{tabular}{c} Jason Miller\\[-5pt]\small Cambridge \end{tabular}
}

\begin{document}

\maketitle

\begin{abstract}
We consider a particular type of $\sqrt{8/3}$-Liouville quantum gravity surface called a doubly marked quantum disk (equivalently, a Brownian disk) decorated by an independent chordal SLE$_6$ curve $\eta$ between its marked boundary points. We obtain descriptions of the law of the quantum surfaces parameterized by the complementary connected components of $\eta([0,t])$ for each time $t \geq 0$ as well as the law of the left/right $\sqrt{8/3}$-quantum boundary length process for $\eta$. 
\end{abstract}

\tableofcontents

\section{Introduction}
\label{sec-intro}

\subsection{Overview} 
\label{sec-overview}

For $\gamma \in (0,2)$, a \emph{$\gamma$-Liouville quantum gravity (LQG) surface} is (formally) the random Riemann surface parameterized by a domain $D\subset \BB C$ whose Riemannian metric tensor is $e^{\gamma h(z)} \,dx\otimes dy$, where $h$ is some variant of the Gaussian free field (GFF) on $D$ and $dx\otimes dy$ is the Euclidean metric tensor. This does not make literal sense since $h$ is a distribution, not a function, so does not take values at points. However, it was shown by Duplantier and Sheffield~\cite{shef-kpz} that one can make sense of the associated volume form $e^{\gamma h(z)} \,dz$ (where $dz$ denotes Lebesgue measure on $D$) associated with a $\gamma$-LQG surface. More precisely, there is a measure~$\mu_h$ on~$D$, called the \emph{$\gamma$-LQG area measure} which is the limit of regularized versions of $e^{\gamma h(z)} \,dz$. A similar construction yields the \emph{$\gamma$-LQG length measure} $\nu_h = ``e^{(\gamma/2) h(z)} \, |dz|"$ which is defined on certain curves in $D$, including $\bdy D$ and SLE$_\kappa$-type curves for $\kappa =\gamma^2$~\cite{shef-zipper}. We remark that there is a more general theory of random measures which have the same law as $\mu_h$, called \emph{Gaussian multiplicative chaos}, which dates back to Kahane~\cite{kahane}; see~\cite{rhodes-vargas-review} for a survey of this theory.

The measures $\mu_h$ and $\nu_h$ are conformally covariant, in the following sense. If $\wt D\subset \BB C$, $f : \wt D \rta D$ is a conformal map, and 
\eqb\label{eqn-lqg-coord}
\wt h := h\circ f + Q\log |f'| \quad \op{for} \quad Q = \frac{2}{\gamma}  +\frac{\gamma}{2} 
\eqe
then $f_* \mu_{\wt h} = \mu_h$ and $f_* \nu_h = \nu_{\wt h}$.  A \emph{$\gamma$-LQG surface} is defined to be an equivalence class of pairs $(D,h)$ where $D\subset \BB C$ and $h$ is a distribution on $D$, with two such pairs declared to be equivalent if they are related by a conformal map as in~\eqref{eqn-lqg-coord} which extends to a homeomorphism $\wt D\cup \bdy \wt D\rta D\cup \bdy D$. Hence a $\gamma$-LQG surface comes equipped with a measure and a conformal structure.

One can also consider $\gamma$-LQG surfaces with $k \in \BB N$ marked points in $D\cup\bdy D$, which are defined in the same manner except that we require the conformal map in~\eqref{eqn-lqg-coord} to map the marked points for one quantum surface to the corresponding marked points for the other.

In the special case when $\gamma =\sqrt{8/3}$, it was shown in~\cite{lqg-tbm1,lqg-tbm2,lqg-tbm3}
that a $\sqrt{8/3}$-LQG surface can also be endowed with a metric space structure. In this case, certain special $\sqrt{8/3}$-LQG surfaces introduced in~\cite{wedges} are equivalent (as metric measure spaces) to \emph{Brownian surfaces}, random metric measure spaces which arise as the scaling limits of uniform random planar maps in the Gromov-Hausdorff topology. In particular, the quantum sphere is equivalent to the Brownian map~\cite{legall-uniqueness,miermont-brownian-map}, the $\sqrt{8/3}$-quantum cone is equivalent to the Brownian plane~\cite{curien-legall-plane}, the quantum disk is equivalent to the Brownian disk~\cite{bet-mier-disk}, and the $\sqrt{8/3}$-quantum wedge is equivalent to the Brownian half-plane~\cite{gwynne-miller-uihpq,bmr-uihpq}.

LQG surfaces arise as the scaling limits of random planar maps. The case $\gamma =\sqrt{8/3}$ corresponds to uniform random planar maps of various types (which is consistent with the equivalence of $\sqrt{8/3}$-LQG and Brownian surfaces); and other values of $\gamma$ correspond to random planar maps sampled with probability proportional to the partition function of a statistical mechanics model. For example, the uniform spanning tree corresponds to $\gamma = \sqrt 2$, the Ising model corresponds to $\gamma = \sqrt 3$, and bipolar orientations correspond to $\gamma= \sqrt{4/3}$.

Many random planar map models are naturally decorated by a statistical mechanics model which can be represented by one or more curves. For many such models, the curve-decorated random planar map converges (or is conjectured to converge) in the scaling limit to a $\gamma$-LQG surface decorated by an independent Schramm-Loewner evolution~\cite{schramm0} (SLE) type curve with parameter\footnote{Here and throughout this paper we use the imaginary geometry~\cite{ig1,ig2,ig3,ig4} convention of writing $\kappa = \gamma^2$ for the SLE parameter when $\kappa \in (0,4)$ and $\kappa' = 16/\kappa = 16/\gamma^2$ for the dual parameter.}
\eqb \label{eqn-gamma-kappa}
\kappa = \gamma^2 \quad \op{or} \quad \kappa' = \frac{16}{\gamma^2} .
\eqe  
See, e.g.,~\cite{shef-burger,kmsw-bipolar,lsw-schnyder-wood} 
for scaling limit results for curved-decorated random planar maps toward SLE-decorated LQG in the so-called \emph{peanosphere sense}; and~\cite{gwynne-miller-saw,gwynne-miller-perc} for scaling limit results in the \emph{Gromov-Hausdorff-Prokhorov-uniform} topology when $\gamma = \sqrt{8/3}$.

In the continuum, there are a number of results which describe the laws of various objects associated with a $\gamma$-LQG surface decorated by an independent SLE$_\kappa$ or SLE$_{\kappa'}$-type curve. For example, the $\gamma$-LQG length measure is defined on SLE$_\kappa$-type curves~\cite{shef-zipper}, a certain particular $\gamma$-LQG surface called a $\gamma$-quantum cone decorated by an independent space-filling SLE$_{\kappa'}$ curve can be encoded by a correlated two-dimensional Brownian motion via the so-called \emph{peanosphere construction}, and the law of the quantum surfaces parameterized by the complementary connected components of certain SLE$_{\kappa}$ or SLE$_{\kappa'}$-type curves on a $\gamma$-LQG surface can be described explicitly~\cite{shef-zipper,wedges,sphere-constructions,gwynne-miller-char}.  Such results are essential for the proofs of scaling limit results for curve-decorated random planar maps toward SLE-decorated LQG, can be used to prove properties of SLE or LQG, and also play a fundamental role in the definition of quantum Loewner evolution~\cite{qle} and thereby the construction of the $\sqrt{8/3}$-LQG metric~\cite{lqg-tbm1,lqg-tbm2,lqg-tbm3}.

The goal of this paper is to prove several results concerning a particular $\sqrt{8/3}$-LQG surface, namely the doubly marked quantum disk (which we recall is equivalent to the Brownian disk with two points sampled uniformly from its boundary length measure~\cite{lqg-tbm2}), decorated by an independent chordal SLE$_6$ between its two marked points. We will describe the laws of the complementary connected components of the curve stopped at any given time as well as the law of the left/right $\sqrt{8/3}$-LQG boundary length process. 

Analogs of our results for chordal SLE$_6$ on a $\sqrt{8/3}$-quantum wedge, the infinite-volume and infinite-boundary length limit of the quantum disk at a marked boundary point, are proven in~\cite{wedges} (in fact, the results of~\cite{wedges} hold for general $\kappa' \in (4,8)$ and $\gamma=4/\sqrt{\kappa'}$).  We will deduce our results from the infinite-volume and infinite-boundary length case and a conditioning argument. 

The results of the present paper complement the characterization theorems for chordal SLE$_6$ on a quantum disk in~\cite{gwynne-miller-char} since they establish that the hypotheses of the characterization theorem are satisfied for chordal SLE$_6$ on a quantum disk and give us an explicit description of the boundary length process appearing in the theorem statement. 
In~\cite{gwynne-miller-perc}, the results of this paper will be used in conjunction with~\cite[Theorem~7.12]{gwynne-miller-char} to identify the scaling limit of a percolation exploration path on a random quadrangulation with simple boundary as SLE$_6$ on an independent quantum disk. 

Chordal SLE$_6$ on an independent doubly marked quantum disk arises naturally in many applications of $\sqrt{8/3}$-LQG. For example, such curves appear when one skips the bubbles filled in by the space-filling SLE$_6$ curve in the peanosphere construction of~\cite{wedges} after a given time $t$~\cite{gwynne-miller-char}; and in the construction of CLE$_6$~\cite{shef-cle} on an independent $\sqrt{8/3}$-LQG surface via branching SLE. 
However, this particular curve-decorated quantum surface is not studied specifically in~\cite{wedges}. 
Hence we expect that the results of the present paper are likely to also have other applications in the study of SLE and $\sqrt{8/3}$-LQG.

\bigskip

\noindent{\bf Acknowledgements}  J.M.\ thanks Institut Henri Poincar\'e for support as a holder of the Poincar\'e chair, during which part of this work was completed. We also thank two anonymous referees for a helpful set of comments which led to many improvements in the exposition of this article.

\subsection{Main results} 
\label{sec-results}

Let $\gamma=\sqrt{8/3}$ and $\kappa'=6$.  Also fix $\frk l^L , \frk l^R > 0$.  Let $\mcl B = (\BB H , h^\nat , 0 , \infty)$ be a doubly marked quantum disk with left/right boundary lengths $\frk l^L$ and $\frk l^R$ (and random area).  That is, $\nu_{h^\nat}((-\infty,0]) = \frk l^L$ and $\nu_{h^\nat}([0,\infty)) = \frk l^R$. See Section~\ref{sec-wedge} for more on the quantum disk.  Let $\eta^\nat$ be a chordal SLE$_6$ from $0$ to $\infty$ in $\BB H$, independent from $h^\nat$, with some choice of parameterization. 

By SLE duality~\cite{zhan-duality1,zhan-duality2,dubedat-duality,ig1,ig4}, for each $u\geq 0$ the boundary of each connected component of $\BB H\setminus \eta^\nat([0,u])$ is the union of a (possibly empty) segment of $\BB R$ and finitely many SLE$_{8/3}$-type curves. Therefore, we can define the $\sqrt{8/3}$-LQG length measure $\nu_{h^\nat}$ on the boundary of each such component using~\cite[Theorem 1.3]{shef-zipper} and local absolute continuity (this measure can also be defined directly as a Gaussian multiplicative chaos measure with respect to the Minkowski content of the SLE$_{8/3}$-type curve using the results of~\cite{benoist-lqg-chaos}).

For $u  \geq 0$, let $L_u^\nat$ (resp.\ $R_u^\nat$) be equal to $-\frk l^L$ (resp.\ $-\frk l^R$) plus the $\nu_{h^\nat}$-length of the boundary arc of the unbounded connected component of $\BB H\setminus \eta^\nat([0,u])$ which lies to the left (resp.\ right) of $\eta^\nat(u)$ and define the \emph{left/right boundary length process} $Z_u^\nat := (L_u^\nat, R_u^\nat)$.  The process $Z^\nat$ is a continuum analog of the left/right boundary length process for the percolation peeling process on a random planar map with boundary (also called the \emph{horodistance}), as studied, e.g., in~\cite{curien-glimpse,gwynne-miller-perc}. Note that the definition of the left/right boundary length process differs from that in~\cite[Section~6]{gwynne-miller-char} by a translation by $(\frk l^L , \frk l^R)$.  The reason for this translation is so that the process in the present paper starts at $(0,0)$. 

The process $L_u^\nat$ (resp.\ $R_u^\nat$) is a right continuous process with no upward jumps, but with a downward jump whenever $\eta^\nat$ disconnects a bubble from $\infty$ on its left (resp.\ right) side. The magnitude of the downward jump corresponds to the boundary length of the bubble.  
 
The main choice of parameterization we will consider for the curve $\eta^\nat$ in this paper is the \emph{quantum natural time} parameterization with respect to $h^\nat$.  Roughly speaking, parameterizing $\eta^\nat$ by quantum natural time with respect to $h^\nat$ is equivalent to parameterizing by the ``quantum local time" at the set of times when $\eta^\nat$ disconnects a bubble from $\infty$. Formally, quantum natural time is defined for a chordal SLE$_{\kappa'}$, $\kappa' \in (4,8)$, and an independent free-boundary GFF on $\BB H$ with a $\gamma/2$-log singularity at the origin in~\cite[Definition~6.23]{wedges}; and is defined for the pair $(  h^\nat , \eta^\nat )$ via local absolute continuity.

In the remainder of this subsection we assume that $\eta^\nat$ is parameterized by quantum natural time with respect to $h^\nat$.  Let $S^\nat$ be the total quantum natural time length of $ \eta^\nat$, so that $ \eta^\nat$ is defined on the random interval $[0,S^\nat]$. We extend $ \eta^\nat$ to all of $[0,\infty)$ by setting $ \eta^\nat(u) = \infty$ for $u > S^\nat$. Note that this implies that $Z^\nat_u = (-\frk l^L,-\frk l^R)$ for $u > S^\nat$.

\begin{thm} \label{thm-sle-bead-nat}
Suppose we are in the setting described just above (so in particular $\gamma=\sqrt{8/3}$, $\kappa'=6$, and $\frk l^L , \frk l^R > 0$). For $u \geq 0$, let $\mcl W^\nat_u$ be the doubly marked quantum surface obtained by restricting $h^\nat$ to the unbounded connected component of $\BB H\setminus \eta^\nat([0,u])$, with marked points $\eta^\nat(u)$ and $\infty$. 
If we condition on $Z^\nat|_{[0,u]}$, then the conditional law of $\mcl W^\nat_u$ is that of a doubly marked quantum disk with left/right boundary lengths $L_u^\nat + \frk l^L$ and $R_u^\nat + \frk l^R$ (here a quantum disk with zero boundary length is a single point). The conditional law of the collection of singly marked quantum surfaces obtained by restricting $h^\nat$ to the bubbles disconnected from $\infty$ by $\eta^\nat$ before time $u$, each marked by the point where $\eta^\nat$ finishes tracing its boundary, is that of a collection of independent singly marked quantum disks parameterized by the downward jumps of the coordinates of $Z^\nat|_{[0,u]}$, with boundary lengths given by the magnitudes of these downward jumps, and these singly marked quantum surfaces are conditionally independent from $\mcl W^\nat_u$.
\end{thm}
 
We also obtain a description of the law of the process $Z^\nat$ in terms of the Radon-Nikodym derivative of the law of $Z^\nat|_{[0,u]}$ with respect to the law of a pair of independent $3/2$ stable processes started from $(0,0)$ (which is the left/right boundary length process for chordal SLE$_6$ on a $\sqrt{8/3}$-quantum wedge~\cite[Corollary~1.19]{wedges}), run up to time $u$. 
Recall that $S^\nat =\inf\{u \geq 0 : \eta^\nat(u) = \infty\} $ is the total quantum natural time of $\eta^\nat$, as defined above the statement of Theorem~\ref{thm-sle-bead-nat}.

\begin{thm} \label{thm-bead-process-law}
Suppose we are in the setting described just above (so in particular $\gamma=\sqrt{8/3}$, $\kappa'=6$, and $\frk l^L , \frk l^R > 0$).
\begin{enumerate}
\item (Endpoint continuity) Almost surely, the time $S^\nat$ is finite. Furthermore, a.s.\ 
\eqb \label{eqn-S^bead-formula}
S^\nat = \inf\left\{ u \geq 0 :  L_u^\nat = -\frk l^L \right\}  = \inf\left\{ u \geq 0 :  R_u^\nat = -\frk l^R \right\}
\eqe 
and a.s.\ $\lim_{u \rta S^\nat}  \eta^\nat(u) = \infty$ and $\lim_{u \rta S^\nat} Z_u^\nat = (-\frk l^L , -\frk l^R)$. \label{item-bead-process-time}
\item (Radon-Nikodym derivative) Let $Z^\wge = ( L^\wge ,   R^\wge)$ be a pair of independent totally asymmetric $3/2$-stable processes with no upward jumps\footnote{The L\'evy measure of each of $L^\wge$ and $R^\wge$ is given by $c |t|^{-5/2}  \BB 1_{(t < 0)} \, dt$ for some constant $c>0$ which we do not compute explicitly. However, it will be clear from the proof that this constant is the same as the (non-explicit) scaling constant for the L\'evy measure of the two coordinates of the stable process as in~\cite[Corollary~1.19]{wedges}.} 
and define
\eqb \label{eqn-S^wge-def}
S^\wge := \inf\left\{ u\geq 0:  L_u^\wge \leq -\frk l^L \: \op{or} \:  R_u^\wge \leq -\frk l^R  \right\}   .
\eqe 
For $u\geq 0$, the sub-probability measure obtained by restricting the law of $Z^\nat|_{[0,u]}$ to the event $\{u < S^\nat\}$ is absolutely continuous with respect to the law of $Z^\wge|_{[0,u]}$, with Radon-Nikodym derivative given by
\eqbn
 \left( \frac{ L_u^\wge +  R_u^\wge}{\frk l^L + \frk l^R} + 1 \right)^{-5/2}\BB 1_{(u < S^\wge) } .
\eqen
\label{item-bead-process-nat} 
\end{enumerate} 
\end{thm}

Assertion~\ref{item-bead-process-time} of Theorem~\ref{thm-bead-process-law} will be proven in Section~\ref{sec-bead-process-time} via a local absolute continuity argument.  In Section~\ref{sec-bead-process-nat}, we will prove Theorem~\ref{thm-sle-bead-nat} and assertion~\ref{item-bead-process-nat} of Theorem~\ref{thm-bead-process-law} simultaneously using the analogous statements for chordal SLE$_6$ on a $\sqrt{8/3}$-quantum wedge (the infinite-volume and infinite-boundary length limit of the quantum disk when we zoom in at a fixed boundary point), which are proven in~\cite{wedges}, and a conditioning argument. 

These statements are proven via a continuum analog of the proof of the Radon-Nikodym derivative estimate~\cite[Lemma~3.6]{gwynne-miller-simple-quad} for peeling processes on finite and infinite quadrangulations with simple boundary, where one peels a single quadrilateral in the uniform infinite half-plane quadrangulation with simple boundary and conditions on the event that the bubble it disconnects from $\infty$ has given boundary length in order to produce a free Boltzmann quadrangulation with simple boundary. The proof is illustrated in Figure~\ref{fig-bead-process-law}.

The idea of the proof of Theorems~\ref{thm-sle-bead-nat} and~\ref{thm-bead-process-law} is as follows.  We start with a $\sqrt{8/3}$-quantum wedge $(\BB H , h^\wge , 0,\infty)$ and an independent chordal SLE$_6$ $\eta^\wge$ from $0$ to $\infty$, parameterized by quantum natural time with respect to $h^\wge$. By~\cite[Corollary~1.19]{wedges} we can take the process $Z^\wge$ appearing in Theorem~\ref{thm-bead-process-law} to be the left/right boundary length process for the pair $(h^\wge , \eta^\wge)$. We then consider a second independent chordal SLE$_6$ curve $\eta^\sfi$ at a point lying at $-\frk l^L$ units of $\nu_{h^\wge}$-length to the left of the origin. We condition on the event $F_\ep(\frk l^L+\frk l^R)$ that the connected component of $\BB H\setminus \eta^\sfi$ containing the origin on its boundary, marked by the point $0$ and the point where $\eta^\sfi$ finishes tracing its boundary, has left/right quantum boundary lengths approximately $\frk l^L$ and $\frk l^R$. The results of~\cite{wedges} imply that on $F_\ep(\frk l^L+\frk l^R)$, the quantum surface parameterized by this connected component is approximately a doubly marked quantum disk with left/right quantum boundary lengths $\frk l^L$ and $\frk l^R$. We deduce our desired results by comparing the (unconditional) probability of $F_\ep(\frk l^L+\frk l^R)$ and the conditional probability of $F_\ep(\frk l^L+\frk l^R)$ given $Z^\wge|_{[0,u]}$ and the quantum surfaces parameterized by the bubbles disconnected from $\infty$ by $\eta^\wge|_{[0,u]}$.

Our proof only works in the special case when $\kappa'=6$ ($\gamma=\sqrt{8/3}$), for the following two reasons.
\begin{enumerate}
\item We make heavy use of the locality property of SLE$_6$. 
\item For $\gamma \in (0,2)$, it is natural to draw a chordal SLE$_{\kappa'}$ on an independent $\frac{4}{\gamma}- \frac{\gamma}{2}$- (weight-$\frac{3\gamma^2}{2}-2$) quantum wedge, which locally looks like a single bead of a $\frac{3\gamma}{2}$-quantum wedge near its marked point~\cite[Theorem~1.18]{wedges}. On the other hand, the $\gamma$-quantum wedge is the only one whose law is invariant under the operation of translating by a given amount of quantum length along the boundary~\cite[Proposition~1.7]{shef-zipper} (this statement is needed since $\eta^\sfi$ is not started at the origin). We have $\frac{4}{\gamma}- \frac{\gamma}{2} = \gamma$ only for $\gamma =\sqrt{8/3}$. 
\end{enumerate}

We now comment briefly on the extent to which we expect our results to extend to other values of $\kappa' \in (4,8)$ and corresponding LQG parameter $\gamma =4/\sqrt{\kappa'} \in (\sqrt 2 , 2)$. For such values of $\kappa'$ we do not expect to get nice descriptions of the law of objects associated with a chordal SLE$_{\kappa'}$ on an independent doubly marked quantum disk. It is instead more natural to consider chordal SLE$_{\kappa'}$ on an independent quantum surface whose law is that of a single bead of a $\frac{3\gamma}{2}$-quantum wedge (see Section~\ref{sec-wedge}). Intuitively, the reason for this is that a quantum disk naturally corresponds to a statistical physics model with either all black or all white boundary conditions.  When $\kappa'=6$, the statistical physics model is percolation so the boundary conditions do not matter.  A bead of a $\frac{3\gamma}{2}$-quantum wedge naturally corresponds to a statistical physics model with white (resp.\ black) boundary conditions on its left (resp.\ right) side.  Mathematically, this manifests itself in that~\cite[Theorem~1.18]{wedges} concerns chordal SLE$_{\kappa'}$ on a $\left( \frac{4}{\gamma} - \frac{\gamma}{2} \right)$-quantum wedge, which is an infinite-volume quantum surface which has the same type of log-singularity at its first marked point as does a single bead of a $\frac{3\gamma}{2}$-quantum wedge. 

We expect that Theorem~\ref{thm-sle-bead-nat} is true for general values of $\kappa'\in (4,8)$ but with a single bead of a $\frac{3\gamma}{2}$-quantum wedge in place of a doubly marked quantum disk, but our proof does not extend to this case. 
Assertion~\ref{item-bead-process-time} of Theorem~\ref{thm-bead-process-law} is true for general values of $\kappa'\in (4,8)$, with the same proof, and remains true if we fix the area of the quantum disk; see Lemma~\ref{lem-terminal-time} below.  
We do not have any reason to believe that a Radon-Nikodym derivative formula as simple as the one in assertion~\ref{item-bead-process-nat} Theorem~\ref{thm-bead-process-law} continues to hold for $\kappa'\not=6$.

\section{Preliminaries}
\label{sec-prelim}

In this section we will introduce some notation and review some facts about Liouville quantum gravity and SLE, mostly coming from~\cite{wedges}, which are needed for the proofs of our results.  The reader who is already familiar with these objects can safely skip this section.

\subsection{Basic notation}
\label{sec-basic}

Here we record some basic notation which we will use throughout this paper. 

\noindent
We write $\BB N$ for the set of positive integers and $\BB N_0 = \BB N\cup \{0\}$. 
\vspace{6pt}

\noindent
For $a < b \in \BB R$, we define the discrete intervals $[a,b]_{\BB Z} := [a, b]\cap \BB Z$ and $(a,b)_{\BB Z} := (a,b)\cap \BB Z$.
\vspace{6pt}

\noindent
If $a$ and $b$ are two quantities, we write $a\preceq b$ (resp.\ $a \succeq b$) if there is a constant $C>0$ (independent of the parameters of interest) such that $a \leq C b$ (resp.\ $a \geq C b$). We write $a \asymp b$ if $a\preceq b$ and $a \succeq b$.
\vspace{6pt}

\noindent
If $a$ and $b$ are two quantities which depend on a parameter $x$, we write $a = o_x(b)$ (resp.\ $a = O_x(b)$) if $a/b \rta 0$ (resp.\ $a/b$ remains bounded) as $x \rta 0$, or as $x\rta\infty$, depending on context. 
\vspace{6pt}

\noindent
Unless otherwise stated, all implicit constants in $\asymp, \preceq$, and $\succeq$ and $O_x(\cdot)$ and $o_x(\cdot)$ errors involved in the proof of a result are required to depend only on the auxiliary parameters that the implicit constants in the statement of the result are allowed to depend on.

\subsection{LQG surfaces} 
\label{sec-wedge}

Recall from Section~\ref{sec-overview} that a $\gamma$-LQG surface for $\gamma \in (0,2)$ with $k\in\BB N$ marked points is an equivalence class $\mcl S$ of $(k+2)$-tuples $(D,h,x_1,\dots ,x_k)$ where $D\subset \BB C$; $h$ is a distribution on $D$, typically some variant of the Gaussian free field (GFF)~\cite{shef-gff,ss-contour,shef-zipper,ig1,ig4}; and $x_1,\dots , x_k \in D\cup\bdy D$ are marked points. Two such $(k+2)$-tuples are equivalent if there is a conformal map $f : \wt D \rta D$ which satisfies~\eqref{eqn-lqg-coord} and also $f(\wt x_j ) = x_j$ for $j\in [1,k]_{\BB Z}$.  We will often slightly abuse notation by writing $\mcl S = (D,h,x_1,\dots ,x_k)$ for a quantum surface (rather than just a single equivalence class representative). 

The main type of LQG surface we will be interested in this paper is the \emph{quantum disk} which is a finite-volume quantum surface (i.e., the total mass of the $\gamma$-LQG area measure $\mu_h$ is finite) which can be represented as $(\BB D , h)$ and is defined formally in~\cite[Section~4.5]{wedges}. One can consider quantum disks with fixed boundary length or with fixed area and boundary length. A \emph{singly (resp.\ doubly) marked quantum disk} is a quantum disk together with one (resp.\ two) marked points sampled uniformly from the $\gamma$-quantum boundary length measure $\nu_h$.  For a doubly marked quantum disk, one can also condition on the left and right quantum boundary lengths, which are the lengths of the left and right boundary arcs between the marked points, respectively, if we stand at the first marked point and look into the interior of the disk. See Appendix~\ref{sec-disk} for a precise definition of the quantum disk.

For our proofs, we will also need to consider several other types of quantum surfaces.  For $\alpha \leq Q$ (with $Q$ as in~\eqref{eqn-lqg-coord}) an \emph{$\alpha$-quantum wedge} is an infinite-volume quantum surface $(\BB H , h^\infty , 0, \infty)$ parameterized by the upper half-plane defined formally in~\cite[Section~1.6]{shef-zipper} or~\cite[Section~4.2]{wedges}. This quantum surface is obtained by starting with a $\wt h^\infty - \alpha\log |\cdot|$, where $\wt h^\infty$ is a free-boundary GFF on $\BB H$; then zooming in near the origin and re-scaling to fix the additive constant in a canonical way~\cite[Proposition~4.7(ii)]{wedges}. 

The case $\alpha=\gamma$ is special since a quantum surface has a $- \gamma$-log singularity at a typical point for the $\gamma$-LQG boundary length measure, so a $\gamma$-quantum wedge describes the local behavior of a $\gamma$-LQG surface at a quantum typical boundary point. The law of $\gamma$-quantum wedge is invariant under translating by a given amount of $\gamma$-quantum length along the boundary. In other words, if we fix $t \geq 0$ and let $x_t$ be chosen so that $\nu_h([0,x_t]) = t$, then $(\BB H , h^\infty , x_t , \infty) \eqD (\BB H , h , 0,\infty)$ as quantum surfaces~\cite[Proposition~1.7]{shef-zipper}.  

For $\alpha \in (Q,Q+\gamma/2)$, an \emph{$\alpha$-quantum wedge} is a Poissonian string of \emph{beads} glued together end-to-end, each of which is itself a finite-volume doubly-marked quantum surface with the topology of the disk; see~\cite[Definition~4.15]{wedges}.  One can also consider a single bead of an $\alpha$-quantum wedge by itself, conditioned on its left and right quantum boundary lengths or on these boundary lengths and its area. The law of such a bead is invariant under the operation of swapping the marked points, and has a $(2Q-\alpha)$-log singularity at each marked point. Quantum wedges for $\alpha \leq Q$ are called \emph{thick} and those for $\alpha \in (Q,Q+\gamma/2)$ are called \emph{thin}. 

The main type of thin quantum wedge in which we will be interested in this paper is the $\frac{3\gamma}{2}$-quantum wedge. The reason for this is that in the special case when $\gamma =\sqrt6$, a single bead of a $\frac{3\gamma}{2} =\sqrt6$-quantum wedge is the same as a doubly marked quantum disk (this follows from the definitions in~\cite[Sections 4.4 and 4.5]{wedges}).  For general $\gamma \in (0,2)$, the laws of a single bead of a $\frac{3\gamma}{2}$-quantum wedge and a $\left( \frac{4}{\gamma} - \frac{\gamma}{2} \right)$-quantum wedge are locally absolutely continuous with respect to each other in a neighborhood of their first marked points.  We note that $\frac{4}{\gamma} - \frac{\gamma}{2} = \gamma$ if and only if $\gamma = \sqrt{8/3}$. 

For $\alpha < Q $, an \emph{$\alpha$-quantum cone} is an infinite volume quantum surface $(\BB C , h' , 0, \infty)$ obtained by starting with a whole-plane GFF plus a $-\alpha$-log singularity at the origin and zooming in near the origin while fixing the additive constant in a consistent way. The $\alpha$-quantum cone is defined in~\cite[Section~4.3]{wedges}. As in the case of quantum wedges, the case $\alpha=\gamma$ is special since a GFF-type distribution has a $-\gamma$-log singularity at a typical point for its $\gamma$-LQG area measure~\cite[Section~3.3]{shef-kpz}.

\subsection{Curve-decorated quantum surfaces} 
\label{sec-surface-curve}

Our main interest in this paper is in quantum surfaces decorated by various types of curves, typically SLE$_{\kappa'}$-type curves for $\kappa' = 16/\gamma^2$ (as in~\eqref{eqn-gamma-kappa}). It will be important for us to distinguish between curves in subsets of $\BB C$ and \emph{curve-decorated quantum surfaces}.  By the latter, we mean an equivalence class of $(k+3)$-tuples $(D ,h , x_1,\dots,x_k , \eta)$ for some $k\in\BB N$, where $(D,h,x_1,\dots,x_k)$ is an equivalence class representative for a quantum surface with $k$ marked points and $\eta$ is a parameterized curve in $D$, with two such $(k+3)$-tuples $(D,h,x_1,\dots,x_k , \eta)$ and $(\wt D, \wt h, \wt x_1,\dots, \wt x_k , \wt\eta)$ defined two be equivalent if there is a conformal map $f : \wt D\rta  D$ which satisfies~\eqref{eqn-lqg-coord} and also $f(\wt x_j) = x_j$ for $j\in [1,k]_{\BB Z}$ and $f\circ\wt\eta = \eta$. 

We allow ``curves" $\eta$ which are defined on a general closed subset of $\BB R$, rather than just a single interval, which arise naturally if we want to restrict a curve to the pre-image of some set whose pre-image is not connected.

Since most of our curves will be originally defined on subsets of $\BB C$, we introduce the following notation.

\begin{defn} \label{def-surface-curve}
Let $ (D,h,x_1,\dots,x_k)$ be an embedding of a quantum surface $\mcl S$ and let $\eta$ be a curve in $\BB C$. Let $I$ be the closure of the interior of $ \eta^{-1}(\ol D) $ and let $\eta_{\mcl S}$ be the curve $\eta|_I$, viewed as a curve on $\mcl S$, so that $(\mcl S , \eta_{\mcl S})$ is a curve-decorated quantum surface represented by the equivalence class of $ (D,h,x_1,\dots,x_k , \eta)$ modulo conformal maps.
\end{defn}

In the remainder of this subsection, we describe how to define a topology on certain families of curve-decorated quantum surfaces using the perspective that a quantum surface is the same as an equivalence class of measure spaces modulo conformal maps.
 
For $k\in\BB N $, let $\BB M_k^{\op{CPU}}$ be the set of all equivalence classes $\mcl K $ of $(k+3)$-tuples $(D , \mu , \eta , x_1,\dots , x_k)$ with $D\subset \BB C$ a simply connected domain (with $D \neq \BB C$), $\mu$ a Borel measure on $D$ which is finite on compact subsets of $ D$, $\eta : \BB R\rta \ol D$ a curve which extends continuously to the extended real line $\BB R\cup \{-\infty ,\infty\}$, $x_1\in D$, and $x_2 ,  \dots , x_k \in D\cup \bdy D$ (with $\bdy D$ viewed as a collection of prime ends).  Two such $(k+3)$-tuples $(D,\mu, \eta , x_1,\dots , x_k)$ and $(\wt D,\wt \mu , \wt\eta, \wt x_1 , \dots , \wt x_k)$ are declared to be equivalent if there exists a conformal map $f : \wt D\rta D$ such that 
\eqb \label{eqn-cp-equiv}
f_* \wt \mu =  \mu ,\quad f \circ \wt \eta = \eta ,  \quad \op{and} \quad f(\wt x_j) =  x_j,\: \forall j\in [1,k]_{\BB Z} .
\eqe  
The reason why $x_1$ is required to be in $ D$ (not in $\bdy D$) is as follows. Just below, we will define a metric on $\BB M_k^{\op{CPU}}$ by ``embedding" two elements of $\BB M_k^{\op{CPU}}$ into the unit disk in such a way that their first marked points are sent to $0$. If we allowed arbitrary embeddings, our metric would not be positive definite due to conformal maps which carry all of the mass to the boundary.  

If $(D , \mu , x_1,\dots ,x_k)$ is a finite measure space (without a curve) with $k$ marked points, the first of which is in the interior, viewed modulo conformal maps, then $(D,\mu,x_1,\dots ,x_k)$ can be viewed as an element of $\BB M_k^{\op{CPU}}$ whose corresponding curve is constant at $x_1$. 
 
The discussion in Section~\ref{sec-wedge} implies that a finite-area curve-decorated $\gamma$-quantum surface $(D,h, \eta , x_1,\dots,x_k)$ with $k$ marked points such that $x_1\in D$ can be viewed as an element of $\BB M_k^{\op{CPU}}$, with $\mu = \mu_h$ the $\gamma$-quantum area measure.   Here we note that the $\gamma$-LQG measure a.s.\ determines the field~\cite{bss-lqg-gff}.  A quantum surface without an interior marked point can be made into a quantum surface with an interior marked point by, e.g., sampling a point uniformly from $\mu_h|_U$ (normalized to be a probability measure) for some open set $U \subset D$ for a fixed choice of equivalence class representative.  We can thereby view such a curve-decorated quantum surface as an element  of $\BB M_{k+1}^{\op{CPU}}$.  Curve-decorated quantum surfaces will be our main examples of elements of $\BB M_k^{\op{CPU}}$. 

To define a metric on $\BB M_k^{\op{CPU}}$, we note that the Riemann mapping theorem implies that each $\mcl K\in \BB M_k^{\op{CPU}}$ admits an equivalence class representative of the form $(\BB D , \mu , \eta ,  0, x_1,\dots,x_k)$ (i.e., the domain is $\BB D$ and the first marked point is~$0$).  For a domain $D\subset \BB C$, let~$\BB d_D^{\op{P}}$ be the Prokhorov metric on finite Borel measures on~$D$ and let $\BB d_D^{\op{U}}$ be the uniform metric on curves in~$D$.  We define the \emph{conformal Prokhorov-uniform distance} between elements $\mcl K  ,\wt{\mcl K} \in \BB M_k^{\op{CPU}}$ by the formula 
\eqb \label{eqn-cp-dist}
\BB d_k^{\op{CPU}}\left(\mcl K , \wt{\mcl K} \right) =  \inf_{\substack{ (\BB D, \mu , \eta  , 0 ,\dots , x_k) \in \mcl K ,\\ (\BB D , \wt\mu  , \wt\eta,  0 , \dots , \wt x_k) \in \wt{\mcl K} } } 
\left\{     \BB d_{\BB D}^{\op{P}}(\mu ,\wt \mu) + \BB d_{\BB D}^{\op{U}}(\eta,\wt\eta) + \sum_{j=2}^k | x_j - \wt x_j|         \right\} .
\eqe  
It is easily verified that $\BB d_k^{\op{CPU}}$ is a metric on $\BB M_k^{\op{CPU}}$, whereby $\mcl K$ and $\wt{\mcl K}$ are $\BB d_{k}^{\op{CPU}}$-close if they can be embedded into $\BB D$ in such a way that their first marked points are sent to $0$, their measures are close in the Prokhorov distance, their curves are close in the uniform distance, and their corresponding marked points are close in the Euclidean distance.  

The above construction gives us a topology on simply connected curve-decorated quantum surfaces. We will also have occasion to consider convergence of beaded quantum surfaces, i.e.\ those which can be represented as a countable ordered collection of finite quantum surfaces, with each such surface attached to its neighbors at a pair of points, with the property that the total quantum area of the beads between any two given beads is finite. One can extend the above metric to curve-decorated beaded quantum surfaces where each bead has $k$ marked points, the first of which is in the interior of the bead, by viewing such a surface as a function from $[0,\infty)$ to $\BB M_k^{\op{CPU}}$ and considering a weighted $L^1$ norm on such functions; see~\cite[Section~2.2.6]{gwynne-miller-char} for more details.

\subsection{Peanosphere construction}  
\label{sec-peanosphere}

In this subsection we review the \emph{peanosphere construction} of~\cite{wedges}, which will occasionally be useful in the proofs of our main theorems (primarily for local absolute continuity purposes). 

\emph{Whole-plane space-filling SLE$_{\kappa'}$ from $\infty$ to $\infty$} is a two-sided variant of SLE$_{\kappa'}$ in $\BB C$ defined in~\cite[Section~1.4.1]{wedges}, building on~\cite[Sections~1.2.3 and~4.3]{ig4}. In the case when $\kappa'  \geq 8$, ordinary SLE$_{\kappa'}$ is already space-filling and whole-plane space-filling SLE$_{\kappa'}$ is just two-sided chordal SLE$_{\kappa'}$. In the case when $\kappa ' \in (4,8)$ (which is the only case we will consider in this paper), whole-plane space-filling SLE$_{\kappa'}$ is obtained from a two-sided variant of chordal SLE$_{\kappa'}$ by iteratively filling in the bubbles surrounded by the curve by space-filling loops. Note that whole-plane space-filling SLE$_{\kappa'}$ in this regime cannot be described via the Loewner evolution. 

Let $\mcl C = (\BB C , h' , 0, \infty)$ be a $\gamma$-quantum cone (recall Section~\ref{sec-wedge}) and let $\eta'$ be a whole-plane space-filling SLE$_{\kappa'}$ with $\kappa' = 16/\gamma^2$ from $\infty$ to $\infty$ independent from $h$ and parameterized in such a way that $\eta'(0) =0$ and the $\gamma$-quantum area measure satisfies $\mu_h(\eta'([s,t])) = t-s$ whenever $s,t\in\BB R$ with $s<t$. 
For $t\geq 0$, let $L_t'$ be equal to the $\gamma$-quantum length of the segment of the left boundary of $\eta'([t,\infty))$ which is shared with $\eta'([0,t])$ minus the $\gamma$-quantum length of the segment of the left boundary of $\eta'([0,t])$ which is not shared with $\eta'([t,\infty))$; and for $t < 0$, let $L_t'$ be the $\gamma$-quantum length of the segment of the left boundary of $\eta'([t,0])$ which is not shared with $\eta'((-\infty,t])$ minus the $\gamma$-quantum length of the segment of the left boundary of $\eta'([t,0])$ which is shared with $\eta'((-\infty,t])$. 
Define $R_t'$ similarly with ``right" in place of ``left". 
Also let $Z' := (L',R') : \BB R\rta \BB R^2$. 

It is shown in~\cite[Theorem~1.9]{wedges} (see also~\cite{kappa8-cov} for the case $\kappa' \geq 8$) that there is a deterministic constant $\alpha = \alpha(\gamma) > 0$ such that $Z'$ evolves as a pair of correlated two-dimensional Brownian motions with variances and covariances given by
\eqbn
\op{Var} L_t' = \op{Var} R_t' = \alpha|t| \quad\op{and}\quad \op{Cov}(L_t',R_t') = -\alpha \cos\left( \frac{\pi \gamma^2}{4} \right) ,\quad \forall t\in\BB R .
\eqen

The Brownian motion $Z$ is referred to as the \emph{peanosphere Brownian motion}. The reason for the name is that $Z$ can be used to construct a random curve-decorated topological space called an \emph{infinite-volume peanosphere}, which a.s.\ differs from $(\BB C ,\eta')$ by a curve-preserving homeomorphism. We will not need this object here so we do not review its definition. 

Henceforth we restrict attention to the case when $\kappa' \in (4,8)$. We will describe how to couple a single bead of a $\frac{3\gamma}{2}$-quantum wedge (with random area and left/right boundary lengths) decorated by a chordal SLE$_{\kappa'}$ with the pair $(h',\eta')$. This coupling is our main use for the peanosphere construction in this paper and is also used in~\cite[Section~6]{gwynne-miller-char}.   See Figure~\ref{fig-surface-def} for an illustration.

\begin{figure}[ht!]
\begin{center}
\includegraphics[scale=.8]{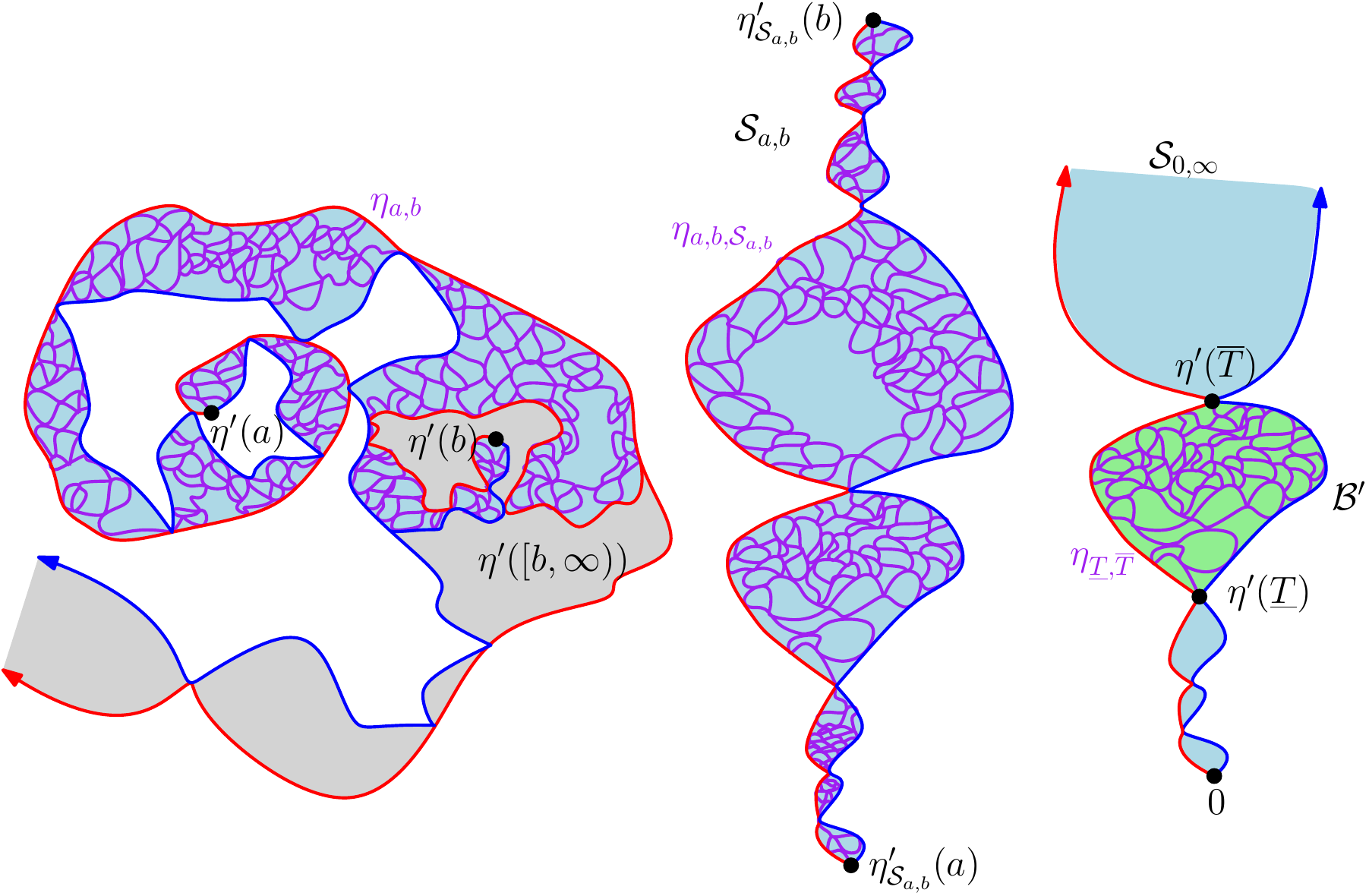}
\end{center}
\caption{
\textbf{Left:} A typical space-filling SLE$_{\kappa'}$ segment $\eta'([a,b])$ for $a < b$ (light blue) and the associated curve $\eta_{a,b}$ (purple).
\textbf{Middle:} The surface $\mcl S_{a,b}$ parameterized by $\eta'([a,b])$ and the curve $\eta_{a,b,\mcl S_{a,b}}$ (which is $\eta_{a,b}$, viewed as a curve on $\mcl S_{a,b}$ instead of in $\BB C$; Definition~\ref{def-surface-curve}). Note that $\mcl S_{a,b}$ does not include information about the exterior pinch points of $\eta'([a,b])$, and its beads correspond to the connected components of the interior of $\eta'([a,b])$.
\textbf{Right:} The $\frac{3\gamma}{2}$-quantum wedge $\mcl S_{0,\infty}$ (light blue) and the single bead $\mcl B' = \mcl S_{\ul T , \ol T}$ (light green) with its associated chordal SLE$_{\kappa'}$ curve $\eta_{\ul T , \ol T  }$ (purple). 
 }\label{fig-surface-def} 
\end{figure}

For $a,b\in \BB R \cup \{-\infty, \infty\}$, let $\mcl S_{a,b}$ be the (possibly beaded) quantum surface parameterized by $\eta'([a,b] )$. Also let $\eta_{a,b} : [a,b] \rta \eta'([a,b])$ be the curve obtained from $\eta'|_{[a,b]}$ by skipping the intervals of time during which it is filling in a bubble (as in~\cite[Section~3.1]{gwynne-miller-char}). Then $\eta_{a,b}$ is a concatenation of SLE$_{\kappa'}$-type curves in the connected components of the interior of $\eta'([a,b])$ (more precisely, $\eta_{a,b}$ in each component is either an SLE$_{\kappa'}$ or an SLE$_{\kappa'}(\kappa'/2-4;\kappa'/2-4)$; see~\cite[Lemma 3.6]{gwynne-miller-char}). Whenever one of these curves disconnects a bubble from its target point, it is constant on an interval of time equal to the $\mu_{h'}$-mass of the bubble. 
 
By the last statement of~\cite[Theorem 1.9]{wedges}, for $t\in\BB R$, the quantum surfaces $\mcl S_{-\infty,t}$ and $\mcl S_{t,\infty}$ are  independent $\frac{3\gamma }{2}$-quantum wedges. Furthermore,~\cite[Footnote 4]{wedges} implies that the future curve $\eta_{t,\infty , \mcl S_{t,\infty} }$ (here we use Definition~\ref{def-surface-curve}) is a concatenation of independent chordal SLE$_{\kappa'}$ curves, one in each of the beads of $\mcl S_{t,\infty}$. 

Let $\BB M$ be the infinite measure on beads of a $\frac{3\gamma}{2}$-quantum wedge (recall~\cite[Definition~4.15]{wedges}) and let $\BB m $ be the infinite measure on $(0,\infty)^3$ which is the pushforward of $\BB M$ under the function which assigns to each bead its vector of area and left/right quantum boundary lengths. 
Let $\frk A \subset (0,\infty)^3$ be a Borel set such that $\BB m(\frk A)$ is finite and positive. 

Let $(\frk a , \frk l^L , \frk l^R)$ be sampled from the probability measure $\BB m(\frk A)^{-1} \BB m|_{\frk A}$, and let $\mcl B'$ be the first bead of the $\frac{3\gamma}{2}$-quantum wedge $\mcl S_{0,\infty}$ whose vector of area and left/right quantum boundary lengths belongs to $\frk A$. 
By our choice of $\frk A$ and since the beads of $\mcl S_{0,\infty}$ are a Poisson point process sampled from $\BB M$, we find that $\mcl B'$ is well-defined a.s.\ and that the law of $\mcl B'$ is that of a single bead of a $\frac{3\gamma}{2}$-quantum wedge conditioned to have area and left/right quantum boundary lengths in $\frk A$. 
Furthermore, if we let $\ul T $ and $\ol T$, respectively, be the times at which $\eta'$ starts and finishes filling in the bead $\mcl B'$, then the curve $\eta_{\ul T , \ol T , \mcl B'}$ (defined as above with $a = \ul T$ and $b=\ol T$ and viewed as a curve on the surface $\mcl B'$ as in Definition~\ref{def-surface-curve}) is a chordal SLE$_{\kappa'}$ in $\mcl B'$ between its two marked points. 

The parameterization of $\eta_{\ul T , \ol T }$ is the parameterization it inherits from $\eta'|_{[\ul T , \ol T]}$. Since $\eta'$ is parameterized by quantum mass, it follows that $\eta_{\ul T , \ol T }$ is parameterized by the $\mu_{h'}$-mass of the region it disconnects from its target point $\eta'(\ol T)$, which is defined as follows (as in~\cite[Definition~6.1]{gwynne-miller-char}).

\begin{defn} \label{def-chordal-parameterization}
Suppose $X$ is a topological space equipped with a measure $\mu$, $x\in X$, and $\eta : [0,T] \rta X$ is a curve with $\eta(T) = x$. We say that $\eta$ is \emph{parameterized by the $\mu$-mass it disconnects from $x$} if the following is true. For $t \in [0,T]$, let $U_t$ be the connected component of $X\setminus \eta([0,t])$ containing $x$ and let $K_t = X\setminus U_t$ be the hull generated by $\eta([0,t])$. Then for each $t\in [0,T]$, 
\eqb  \label{eqn-chordal-parameterization}
\eta(t) = \eta\left( \inf\left\{ s\in [0,T] : \mu(K_s) \geq t\right\} \right) .
\eqe
\end{defn} 

The left/right boundary length process of $\eta_{\ul T , \ol T}$ (as defined in Section~\ref{sec-results}) is obtained from $(Z-Z_{\ul T})|_{[\ul T , \ol T]} $ by skipping the intervals of time during which $\eta_{\ul T , \ol T}$ is filling in one of the bubbles which it disconnects from $\eta'(\ol T)$ (equivalently the maximal $\pi/2$-cone times for $Z$ in the time interval $[\ul T , \ol T]$, defined as in~\cite[Section~2.3.4]{gwynne-miller-char}).

\section{Proofs of main results}
\label{sec-bead-process-law}

\subsection{Proof of assertion~\ref{item-bead-process-time} of Theorem~\ref{thm-bead-process-law}}
\label{sec-bead-process-time}

In this subsection we will prove assertion~\ref{item-bead-process-time} of Theorem~\ref{thm-bead-process-law}. As promised after the statement of Theorem~\ref{thm-bead-process-law}, we do this for general values of $\gamma \in (\sqrt 2 , 2)$ and also allow for fixed area. 

\begin{lem} \label{lem-terminal-time}
Let $\gamma \in (\sqrt 2 , 2)$, let $\kappa' = 16/\gamma^2$, and fix $\frk a, \frk l^L , \frk l^R >0$. Let $\mcl B = (\BB H ,h^\nat  , 0,\infty)$ be a single bead of a $\frac{3\gamma}{2}$-quantum wedge with area $\frk a$ and left/right boundary lengths $\frk l^L$ and $\frk l^R$ and let $ \eta^\nat$ be a chordal SLE$_{\kappa'}$ from 0 to $\infty$ in $\BB H$, sampled independently from $h^\nat$ and then parameterized by quantum natural time with respect to $h^\nat$.  Let $Z^\nat $ be the left/right boundary length process for $\eta^\nat$ started from $(0,0)$ and let $S^\nat$ be the total quantum natural time length of $ \eta^\nat$, as in Section~\ref{sec-results}.

Almost surely, the time $S^\nat$ is finite. Furthermore, a.s.\ 
\eqb \label{eqn-S^bead-formula'}
S^\nat = \inf\left\{ u \geq 0 : L_u^\nat = -\frk l^L \right\}  = \inf\left\{ u \geq 0 : L_u^\nat = -\frk l^R \right\}
\eqe 
and a.s.\ $\lim_{u \rta S^\nat}  \eta^\nat(u) = \infty$ and $\lim_{ u \rta S^\nat}  Z_u^\nat = (-\frk l^L , -\frk l^R)$.  
\end{lem}
\begin{proof} 
By local absolute continuity of the law of $h^\nat$ with respect to the law of an embedding into $(\BB H , 0, \infty)$ of a $\left(\frac{4}{\gamma} - \frac{\gamma}{2} \right)$-quantum wedge (Section~\ref{sec-wedge}), we infer that the quantum natural time length of any segment of the curve $ \eta^\nat$ before the first time it reaches $\infty$ is finite. Equivalently, $\lim_{u \rta S^\nat}  \eta^\nat(u) = \infty$. For each $r > 0$, a.s.\ $ \eta^\nat$ hits each of $(-\infty,-r]$ and $[r,\infty)$ before hitting $\infty$. Since the $\nu_{h^\nat}$-length of any non-empty interval in $\BB R$ is positive and $\nu_{h^\nat}$ does not have any atoms, we obtain the formula~\eqref{eqn-S^bead-formula'}.

To show that $S^\nat$ is a.s.\ finite, we need to show that the quantum natural time length of the final segment of $ \eta^\nat$ is finite. 
Let $\eta_\infty^\nat$ be the time reversal of $ \eta^\nat$, parameterized by the $\mu_{  h^\nat}$-mass of the region it disconnects from $0$. By reversibility of SLE$_{\kappa'}$~\cite{ig3}, $ \eta_\infty^\nat$ is a chordal SLE$_{\kappa'}$ from $\infty$ to 0. By~\cite[Definition~4.15]{wedges} and reversibility of Bessel excursions, we also have $(\BB H ,h^\nat , 0,\infty) \eqD (\BB H , h^\nat , \infty ,0)$.
For $u > 0$, the total quantum natural time length of $ \eta^\nat|_{[u,\infty)} $ is the same as the total quantum natural time length of an appropriate segment of $ \eta_\infty^\nat$ before it hits $0$, so is a.s.\ finite by the first paragraph. Hence a.s.\ $S^\nat <\infty$.  

It remains to check that $\lim_{u \rta S^\nat}   Z_u^\nat = (-\frk l^L , -\frk l^R)$. For this purpose, let $\eta^\bead$ be given by $\eta^\nat$, parameterized by the $\mu_{h^\nat}$-mass it disconnects from $\infty$. 
If we sample $(\frk a , \frk l^L , \frk l^R)$ randomly from a set $\frk A\subset (0,\infty)^3$ as in Section~\ref{sec-peanosphere}, then we can couple $(\mcl B , \eta^\bead_{\mcl B})$ with a $\gamma$-quantum cone and an independent whole-plane space-filling SLE$_{\kappa'}$ as in that subsection in such a way that $(\mcl B , \eta^\bead_{\mcl B}) = (\mcl B' , \eta_{\ul T ,\ol T ,\mcl B'})$ a.s., where $\eta_{\ul T ,\ol T,\mcl B'} $ is the ordinary SLE$_{\kappa'}$-type curve in the bead $\mcl B'$, as defined at the end of Section~\ref{sec-peanosphere}. Continuity of the peanosphere Brownian motion $Z'$ implies that the boundary length process $Z^\bead$ for $\eta^\bead$ a.s.\ satisfies $\lim_{t\rta \frk a} Z_t^\bead = (0,0)$, so since $ Z^\nat$ is a re-parameterization of $Z^\bead $, a.s.\ $\lim_{u \rta S^\nat }  Z_u^\nat = (-\frk l^L , -\frk l^R)$. 
The statement for fixed $(\frk a ,\frk l^L , \frk l^R)$ follows since for different choices of $(\frk a ,\frk l^L , \frk l^R)$, the laws of the fields $h^\nat|_{\BB H \setminus \BB D}$ are mutually absolutely continuous provided we choose our embedding in such a way that the quantum areas of these restricted fields agree (this can be seen from the definition of a beaded quantum wedge in~\cite[Definition 4.15]{wedges} and standard local absolute continuity lemmas for the GFF, e.g.,~\cite[Proposition 3.4]{ig1}).    
\end{proof}

\subsection{Proof of Theorem~\ref{thm-sle-bead-nat} and assertion~\ref{item-bead-process-nat} of Theorem~\ref{thm-bead-process-law}}
\label{sec-bead-process-nat}
 
Throughout this subsection, we assume we are in the setting of Section~\ref{sec-results}, so that $\gamma =\sqrt{8/3}$, $\kappa'=6$, $\frk l^L , \frk l^R > 0$, $\mcl B = (\BB H , h^\nat , 0, \infty)$ is a doubly marked quantum disk with left/right quantum boundary lengths $\frk l^L$ and $\frk l^R$, $\eta^\nat$ is chordal SLE$_6$ from 0 to $\infty$ in $\BB H$, sampled independently from $h^\nat$ and parameterized by quantum natural time with respect to $h^\nat$, $Z^\nat = ( L^\nat , R^\nat)$ is the corresponding left/right boundary length process, and $S^\nat$ is the total quantum natural time length of $\eta^\nat$. 
In this subsection we will simultaneously prove Theorem~\ref{thm-sle-bead-nat} and assertion~\ref{item-bead-process-nat} of Theorem~\ref{thm-bead-process-law}.  
For this purpose we will consider the following setup. 
  
For $u \geq 0$, let $  K_u^\nat \subset \ol{\BB H}$ be the hull which is the closure of the set of points in $\BB H$ which are disconnected from $\infty$ by $ \eta^\nat([0,u])$ and define the quantum surfaces
\eqb \label{eqn-bead-hull-surface}
\mcl K_u^\nat := \left(  K_u^\nat , h^\nat|_{ K_u^\nat} , 0  , \eta^\nat(u)   \right)   \quad \op{and} \quad
\mcl W_u^\nat := \left( \BB H\setminus K_u^\nat , h^\nat |_{\BB H\setminus K_u^\nat} , \eta^\nat(u) , \infty \right) .
\eqe 
The surface $\mcl K_u^\nat$ is typically a beaded quantum surface, since $\eta^\nat$ has cut points (these cut points correspond to the marked points of the beads). The surface $\mcl W_u^\nat$ has the topology of the disk and is the same as the quantum surface defined in Theorem~\ref{thm-sle-bead-nat}.  See Figure~\ref{fig-sle6-surface} for an illustration of these surfaces.

\begin{figure}[ht!]
\begin{center}
\includegraphics[scale=.8]{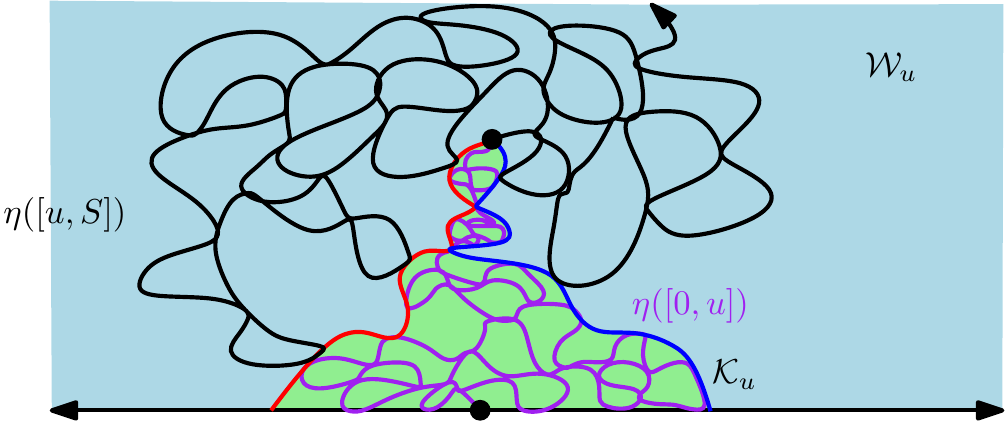}
\end{center}
\caption{
Illustration of the surfaces $\mcl K_u^\nat$ (green) and $\mcl W_u^\nat$ (blue) defined in~\eqref{eqn-bead-hull-surface}. The left and right outer boundaries of the purple curve segment $\eta^\nat([0,u])$ are shown in red and blue, respectively.
 }\label{fig-sle6-surface} 
\end{figure}

We observe that $Z^\nat|_{[0,u]}$ is determined by the curve-decorated quantum surface $( \mcl K_u^\nat,  \eta^\nat_{ \mcl K_u^\nat})$ (we can determine the arc of $\bdy \mcl K_u^\nat$ corresponding to $\bdy  K_u^\nat \setminus \BB R$ since this is the arc whose points are all hit by $\eta^\nat$).
Our goal is to describe the laws of $\mcl W_u^\nat$ and $(  \mcl K_u^\nat, \eta^\nat_{\mcl K_u^\nat})$ (and hence also the law of $Z^\nat$). The latter law will be described by comparison to the analogous objects for an SLE$_6$ on a $\sqrt{8/3}$-quantum wedge. 

To this end, let $(\BB H ,h^\wge , 0, \infty)$ be a $\sqrt{8/3}$-quantum wedge and let $\eta^\wge$ be a chordal SLE$_6$ from~$0$ to $\infty$ in~$\BB H$ parameterized by $\sqrt{8/3}$-quantum natural time with respect to~$h^\wge$.  Let $Z^\wge = (L^\wge , R^\wge) : [0,\infty) \rta \BB R^2$ be the left/right quantum boundary length process for $h^\wge$, defined in an analogous manner to the process $Z^\nat$ above. By~\cite[Corollary~1.19]{wedges}, $Z^\wge$ is a pair of independent totally asymmetric $3/2$-stable processes with no upward jumps.  We define the time $S^\wge$ as in~\eqref{eqn-S^wge-def} for the process $Z^\wge$. 
 
As in~\eqref{eqn-bead-hull-surface}, for $u \geq 0$ let $  K_u^\wge \subset \ol{\BB H}$ be the hull which is the closure of the set of points in $\BB H$ which are disconnected from $\infty$ by $\eta^\wge([0,u])$ and in analogy with~\eqref{eqn-bead-hull-surface} define the quantum surfaces 
\eqb \label{eqn-wge-hull-surface}
\mcl K_u^\wge := \left(  K_u^\wge , h^\wge|_{ K_u^\wge} , 0 , \eta^\wge(u)   \right)  \quad \op{and} \quad
\mcl W_u^\wge := \left(  \BB H\setminus K_u^\wge , h^\wge|_{ \BB H\setminus K_u^\wge} , \eta^\wge(u) , \infty   \right)  .
\eqe 
The main goal of the remainder of this subsection is to prove the following proposition, which will immediately imply both Theorem~\ref{thm-sle-bead-nat} and assertion~\ref{item-bead-process-nat} of Theorem~\ref{thm-bead-process-law}.

\begin{thm} \label{thm-wge-bead-rn}
For each $u \geq 0$, the law of the curve-decorated quantum surface $(\mcl K_u^\nat , \eta^\nat_{\mcl K_u^\nat})$ restricted to the event $  \{ u  < S^\nat\}$ (i.e., the event that $\eta^\nat$ has not reached $\infty$ by time $u$) is absolutely continuous with respect to the law of $(\mcl K_u^\wge , \eta^\nat_{\mcl K_u^\wge})$, with Radon-Nikodym derivative given by
\eqb \label{eqn-wge-bead-rn}
  \left( \frac{L_u^\wge + R_u^\wge}{\frk l^L + \frk l^R} + 1 \right)^{-5/2}\BB 1_{(u < S^\wge) }  .
\eqe 
Furthermore, the conditional law of $\mcl W_u^\nat$ given $(\mcl K_u^\nat , \eta^\nat_{\mcl K_u^\nat})$ is that of a doubly marked quantum disk with left/right boundary lengths $L_u^\nat + \frk l^L$ and $R_u^\nat + \frk l^R$.  
\end{thm}

Before giving the proof of Theorem~\ref{thm-wge-bead-rn}, we use it to deduce Theorems~\ref{thm-sle-bead-nat} and~\ref{thm-bead-process-law}. 

\begin{proof}[Proof of Theorem~\ref{thm-sle-bead-nat} assuming Theorem~\ref{thm-wge-bead-rn}]
By~\cite[Theorem~1.18]{wedges} we know that the conditional law given $Z^\wge |_{[0,u]}$ of the singly marked quantum surfaces parameterized by the bubbles disconnected from $\infty$ by $\eta^\wge([0,u])$ is that of a collection of independent quantum disks with boundary lengths given by the magnitudes of the downward jumps of the coordinates of $Z^\wge$ up to time $u$. By combining this with Theorem~\ref{thm-wge-bead-rn}, we obtain that the conditional law given $Z^\nat |_{[0,u]}$ of the singly marked quantum surfaces parameterized by the bubbles disconnected from $\infty$ by $\eta^\nat([0,u])$ is that of a collection of independent quantum disks with boundary lengths given by the magnitudes of the downward jumps of the coordinates of $Z^\nat$ up to time $u$. This together with the last statement of Theorem~\ref{thm-wge-bead-rn} yields Theorem~\ref{thm-sle-bead-nat}. 
\end{proof}

\begin{proof}[Proof of Theorem~\ref{thm-bead-process-law} assuming Theorem~\ref{thm-wge-bead-rn}]
Assertion~\ref{item-bead-process-time} was proven in Lemma~\ref{lem-terminal-time}. Since $Z^\nat|_{[0,u]}$ on the event $\{S^\nat  <u\}$ and $Z^\wge|_{[0,u]}$ on the event $\{S^\wge < u\}$ are given by the same deterministic functional of the curve-decorated quantum surfaces $(\mcl K_u^\nat , \eta^\nat_{\mcl K_u^\nat})$ and $(\mcl K_u^\wge , \eta^\nat_{\mcl K_u^\wge})$, respectively (here we note that a L\'evy process is a.s.\ determined by its jumps), assertion~\ref{item-bead-process-nat} is immediate from Theorem~\ref{thm-wge-bead-rn}.  
\end{proof}

\begin{figure}[ht!]
\begin{center}
\includegraphics[scale=.8]{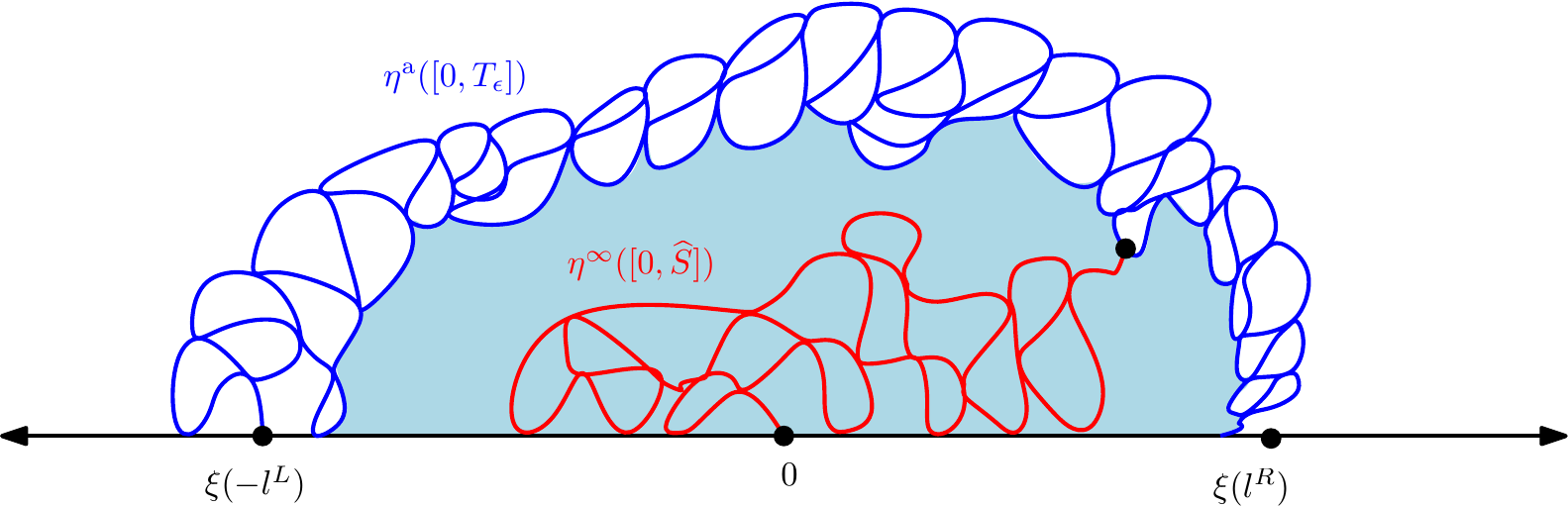}
\end{center}
\caption{
Illustration of the proof of Theorem~\ref{thm-wge-bead-rn}. On the event $F_\ep(\frk l^L + \frk l^R)$, the auxiliary SLE$_6$ curve $\eta^\sfi$ (blue) forms a bubble $D_\ep$ (light blue region) which disconnects 0 from $\infty$ such that $\bdy D_\ep\cap \BB R$ is close to the interval $[\xi(-\frk l^L) , \xi(\frk l^R)]$ which has $\frk l^L$ units of $\nu_{h^\wge}$-length to the left of 0 and $\frk l^R$ units of $\nu_{h^\wge}$-length to the right of 0; and $\nu_{h^\wge}(\bdy D_\ep \setminus \BB R)$ is small. 
In the limit as $\ep\rta 0$, the quantum surface parameterized by $D_\ep$ and the SLE$_6$ curve $\eta^\wge$ (red), stopped at the first time $\wh S$ it hits $\eta^\sfi([0,T_\ep])$, converge in law to the curve-decorated quantum surface $(\mcl B, \eta^\nat_{\mcl B})$ appearing in Theorem~\ref{thm-wge-bead-rn} in a rather strong sense (Lemma~\ref{lem-cont-peel-conv}). 
Since the  left/right boundary length process $Z^\sfi$ for $\eta^\sfi$ evolves as a pair of independent $3/2$-stable processes with no upward jumps, we can compute the probability of $F_\ep(\frk l^L +\frk l^R)$ and the conditional probability of $F_\ep(\frk l^L + \frk l^R) \cap \{\wh S < u\}$ given the curve-decorated quantum surface $(\mcl K_u^\wge , \eta^\nat_{\mcl K_u^\wge})$ corresponding to an initial segment of $\eta^\wge$ (Lemma~\ref{lem-cont-peel-prob}). From these computations together with Bayes' rule, we obtain the Radon-Nikodym derivative of the conditional law of $(\mcl K_u^\wge , \eta^\nat_{\mcl K_u^\wge})$ given $F_\ep(\frk l^L +\frk l^R) \cap \{\wh S <u\}$ with respect to the marginal law of $(\mcl K_u^\wge , \eta^\nat_{\mcl K_u^\wge})$ (Lemma~\ref{lem-cont-peel-rn}). Taking a limit as $\ep\rta0$ yields the theorem statement. 
 }\label{fig-bead-process-law} 
\end{figure} 

\subsubsection{Auxiliary SLE$_6$ curve and disconnection event}

In the remainder of this subsection we give the proof of Theorem~\ref{thm-wge-bead-rn}. See Figure~\ref{fig-bead-process-law} for an illustration. 

Let $(\BB H , h^\wge , 0,\infty)$ be the $\sqrt{8/3}$-quantum wedge from above. 
Let $\xi : \BB R\rta \BB R$ be the increasing function which parameterizes $\BB R$ according to quantum length with respect to $h^\wge$ with $\xi(0) = 0$, so that $\nu_{h^\wge}(\xi([a,b])) = b-a$ for each $a,b\in\BB R$ with $a<b$. 

The idea of the proof of Theorem~\ref{thm-wge-bead-rn} is, roughly speaking, to condition on the ``event" that the points $\xi(-\frk l^L)$ and $\xi(\frk l^R)$ lie at distance zero from one another in a certain quantum sense, which gives us a bubble whose boundary is $\xi([-\frk l^L , \frk l^R])$ and which parameterizes a quantum surface with the law of $\mcl B$. 

More precisely, we will consider an auxiliary SLE$_6$ curve $\eta^\sfi$ started from the point $\xi(-\frk l^L)$ and condition on a certain event $F_\ep(\frk l^L + \frk l^R)$ on which the intersection with~$\BB R$ of the bubble cut off from $\infty$ by $\eta^\sfi$ is close to $\xi([-\frk l^L, \frk l^R])$ and the quantum length of the boundary segment of this bubble which is not part of~$\BB R$ is small. Taking a limit as $\ep\rta 0$ produces a quantum surface with the law of $\mcl B$.  The locality property of SLE$_6$ and together with~\cite[Theorem~1.18]{wedges} will allow us to compute the approximate conditional probability of $F_\ep(\frk l^L + \frk l^R)$ given $(\mcl K_u^\wge , \eta^\wge_{\mcl K_u^\wge})$ and $  \{ u  < S^\nat\}$ for every $u \geq 0$; this computation together with an application of Bayes' rule will imply Theorem~\ref{thm-wge-bead-rn}. 
 
We now proceed with the proof of Theorem~\ref{thm-wge-bead-rn}. We start by introducing our auxiliary SLE$_6$ and the event we will condition on.  Let the $\sqrt{8/3}$-quantum wedge $(\BB H , h^\wge , 0, \infty)$, the SLE$_6$ curve $\eta^\wge$, and the boundary length parameterization $\xi$ be as above.  Conditional on $h^\wge$ and $\eta^\wge$, let $\eta^\sfi$ be a chordal SLE$_6$ from $\xi(-\frk l^L)$ to~$\infty$ in~$\BB H$, parameterized by quantum natural time with respect to~$h^\wge$ and let $Z^\sfi = (L^\sfi , R^\sfi)$ be its left/right boundary length process, so that by~\cite[Corollary~1.19]{wedges}, $Z^\sfi$ is a pair of independent totally asymmetric $3/2$-stable processes with no upward jumps (the $\sfi$ stands for ``auxiliary"). 
For $u \geq 0$, define
\eqbn
 L_{u^-}^\sfi  := \lim_{v\rta u^-} L_v^\sfi  \quad \op{and} \quad  R_{u^-}^\sfi := \lim_{v\rta u^-} R_v^\sfi .
\eqen

For $\ep >  0$, let $T_\ep $ be the first time at which $\eta^\sfi$ disconnects a bubble from $\infty$ to its right with quantum boundary length at least $\ep$, equivalently 
\eqb \label{eqn-sfi-bubble-time}
T_\ep = \inf\left\{ u \geq 0 : R_u^\sfi - R_{u^-}^\sfi \leq -\ep \right\} .
\eqe 
Also let $D_\ep$ be the bubble disconnected from $\infty$ by $\eta^\sfi$ at time $T_\ep$, so that $\nu_{h^\wge}(\bdy D_\ep) =   R_{T_\ep^-}^\sfi - R_{T_\ep}^\sfi $.

Fix a small parameter $\zeta \in (0,1) $ and for $\ep > 0$ and $r>0$, define events 
\eqb \label{eqn-cont-peel-event}
F_\ep^0  := \left\{ \sup_{u \in [0,T_\ep)} | R_u^\sfi | \leq \ep^{1-\zeta}   \right\} \quad \op{and} \quad
F_\ep(r) := F_\ep^0  \cap \left\{  -r -\ep \leq R_{T_\ep}^\sfi - R_{T_\ep^-}^\sfi   \leq  -r   \right\} .
\eqe 
On $F_\ep (r)$, the curve $\eta^\sfi$ does not disconnect $\xi(-\frk l^L + \ep^{1-\zeta})$ from $\infty$ before time $T_\ep$, and the quantum length of the right outer boundary of $\eta^\sfi([0,u])$ for $u  < T_\ep$ is given by $R_u^\sfi - \inf_{v \in [0,u]} R_v^\sfi \leq 2\ep^{1-\zeta}$. Consequently, on $F_\ep(r)$ it holds that  
\eqb \label{eqn-cont-peel-property}
\nu_h(\bdy D_\ep) \in [r , r + \ep  ] ,\quad 
\left[\xi(-\frk l^L +\ep^{1-\zeta}) , \xi(r-\frk l^L  - 2\ep^{1-\zeta} ) \right] \subset \bdy D_\ep , \quad \op{and} \quad 
\nu_h(\bdy D_\ep \setminus \BB R)   \leq 2\ep^{1-\zeta} .
\eqe

We will now estimate the probability of $F_\ep(r)$. 

\begin{lem} \label{lem-cont-peel-prob}
We have
\eqbn
\BB P\left[ F_\ep(r)  \right] =  \left(\frac32 + o_\ep(1) \right) \ep^{5/2} r^{-5/2} 
\eqen
with the rate of the $o_\ep(1)$ uniform for $r\geq r_0$ for any fixed $ r_0$. 
\end{lem}
\begin{proof} 
Recall that the jumps of $-R^\sfi$ are distributed as a Poisson point process on $\BB R$, with intensity measure $c y^{-5/2} \BB 1_{(y>0)} \, dy$, for $c>0$ a universal constant. 
The conditional law of the jump $R^\sfi_{T_\ep^-}  -R^\sfi_{T_\ep} $ given $\{R^\sfi_u : u < T_\ep^-\}$ is given by conditioning this intensity measure on the set $\left[\ep      ,\infty\right)$, i.e.\ this conditional law is given by 
\eqbn
\BB P\left[ R^\sfi_{T_\ep^-} -R^\sfi_{T_\ep} \in \, dy  \,|\, \{R^\sfi_u : u < T_\ep^-\} \right] =     \frac32 \ep^{ 3/2}  y^{-5/2} \BB 1_{\left( y \geq  \ep  \right) } \, dy  .
\eqen
Since the event $F_\ep^0$ is determined by $\{R^\sfi_u : u < T_\ep^-\} $, 
\alb
\BB P\left[ F_\ep(r) \,|\, F_\ep^0 \right] 
 =   \ep^{ 3/2} \left( r^{- 3/2}  -  (r+\ep)^{-3/2}   \right) 
 =  \left(\frac32 + o_\ep(1) \right) \ep^{5/2}  r^{-5/2} .
\ale 

We will complete the proof by showing that there is a universal constant $a >0$ such that 
\eqb \label{eqn-stable-initial}
\BB P\left[ F_\ep^0 \right] \geq 1 - e^{-a \ep^{-\zeta} }. 
\eqe  
To this end, let $\tau_0 =0$ and inductively for $k\in\BB N$ let $\tau_k$ be the smallest $u \geq \tau_{k-1}$ for which $|R^\sfi_u - R_{\tau_{k-1}}^\sfi| \geq \ep$. By the strong Markov property, the increments $(R^\sfi - R_{\tau_{k-1}}^\sfi)|_{[\tau_{k-1} , \tau_k]}$ for $k\in\BB N$ are i.i.d. Let $K$ be the smallest $k\in\BB N$ for which there exists $u\in [\tau_{k-1} , \tau_k]$ for which $R^\sfi_u   - R^\sfi_{u^-} \leq -\ep$. Then $T_\ep \in [\tau_{K-1} , \tau_K]$ and since $\sup_{u \in [\tau_{k-1} , \tau_k]} |R^\sfi_u - R^\sfi_{\tau_{k-1}}| \leq 2\ep$ for $k < K$, 
\eqb \label{eqn-stable-K}
\sup_{u \in [0,T_\ep)} |R^\sfi_u| \leq 2K \ep.
\eqe  
There is a $p> 0$ such that $\BB P\left[ K = k \,|\, k \geq K \right] \geq p$ for each $k\in \BB N_0$. Hence $K$ is stochastically dominated by a geometric random variable, so there exists $a > 0$ such that $\BB P[ K > (1/2) \ep^{-\zeta}] \leq e^{-a\ep^{-\zeta}}$. The bound~\eqref{eqn-stable-initial} therefore follows from~\eqref{eqn-stable-K}. 
\end{proof}

\subsubsection{Convergence of conditional laws given $F_\ep(\frk l^L+\frk l^R)$}

Let
\eqbn
\wh S := \inf\left\{ u \geq 0 : \eta^\wge(u) \in \eta^\sfi([0,\infty)) \right\} .
\eqen
On the event $F_\ep(\frk l^L + \frk l^R)$ (defined as in~\eqref{eqn-cont-peel-event} with $r = \frk l^L + \frk l^R$), let $D_\ep$ be the region disconnected from $\infty$ by $\eta^\sfi$ at time $T_\ep$ (as above). Define the doubly marked quantum surface
\eqbn
\mcl B_\ep = \left( D_\ep  , h^\wge|_{D_\ep} , 0, \eta^\sfi(T_\ep)  \right)  .
\eqen
Recall from~\eqref{eqn-wge-hull-surface} the hulls $K_u^\wge$ generated by $\eta^\wge([0,u])$ for $u \in [0,\wh S)$ define doubly marked quantum surfaces 
\eqb \label{eqn-cond-bubble-surface}
\mcl W_{\ep,u} := \left( D_\ep \setminus K_u^\wge , h^\wge |_{ D_\ep \setminus K_u^\wge} , \eta^\wge(u) , \eta^\sfi(T_\ep) \right)  .
\eqe 
Let $\mcl W_{\ep,u}$ be the degenerate single-point quantum surface for $u \geq \wh S$. 

The main reason for our interest in the quantum surfaces defined above is the following lemma.

\begin{lem} \label{lem-cont-peel-conv}
For any finite collection of times $u_1 ,\dots , u_n \geq 0$, the joint conditional law given $F_\ep(\frk l^L + \frk l^R)$ of the curve-decorated quantum surface $\left(\mcl B_\ep , \eta_{\mcl B_\ep}^\wge|_{[0,\wh S ]} \right)$, the process $Z^\wge|_{[0, \wh S ]}$, the quantum surfaces $\{\mcl W_{\ep,u_k \wedge \wh S}\}_{k\in [1,n]_{\BB Z}}$, and the curve-decorated quantum surfaces $\left\{ \left( \mcl K_{u_k \wedge \wh S}^\wge , \eta^\wge_{ \mcl K_{u_k \wedge \wh S}^\wge}  \right) \right\}_{k\in [1,n]_{\BB Z}}$ defined as in~\eqref{eqn-wge-hull-surface} (with each bead of each quantum surface equipped with an extra interior marked point sampled uniformly from its quantum measure) converges as $\ep\rta 0$ to the joint law of $(\mcl B , \eta_{\mcl B}^\nat)$, the process $Z^\nat|_{[0,S^\nat]}$, the quantum surfaces $\{\mcl W_{u_k}\}_{k\in [1,n]_{\BB Z}}$, and the curve-decorated quantum surfaces $\left\{ \left( \mcl K_{u_k }^\nat , \eta^\nat_{ \mcl K_{u_k}^\nat}  \right) \right\}_{k\in [1,n]_{\BB Z}}$ (with each bead of each surface equipped with an extra interior marked point sampled uniformly from its quantum measure) defined as in the beginning of this subsection. 
The topology of convergence is given by the topology on beaded curve-decorated quantum surfaces, as discussed at the end of Section~\ref{sec-surface-curve}, and the Skorokhod topology, as appropriate. 
\end{lem}

For the proof of Lemma~\ref{lem-cont-peel-conv}, we first need a description of the conditional law of curve-decorated quantum surface $\left(\mcl B_\ep , \eta_{\mcl B_\ep}^\wge|_{[0,\wh S ]} \right)$ given the auxiliary boundary length process $Z^\sfi$ on the event $F_\ep(\frk l^L + \frk l^R)$. 

\begin{lem} \label{lem-aux-surface-law}
On the event $F_\ep(\frk l^L + \frk l^R)$ of~\eqref{eqn-cont-peel-event}, the following holds.
\begin{enumerate}
\item\label{item-aux-surface-full} The conditional law of the quantum surface $\mcl B_\ep$ given $Z^\sfi$ is that of a doubly marked quantum disk with left/right boundary lengths $\frk l^L + R^\sfi_{T_\ep^-}$ and $-R^\sfi_{T_\ep} - \frk l^L $. 
\item\label{item-aux-surface-curve} The conditional law of the curve $\eta^\wge|_{[0,\wh S ]}$ given $Z^\sfi$ and $h^\wge$ is that of a chordal SLE$_6$ from 0 to $\eta^\sfi(T_\ep)$ in $D_\ep$, stopped at the first time it hits $\bdy D_\ep\setminus \BB R$. 
\end{enumerate} 
\end{lem}
\begin{proof}
By~\cite[Theorem~1.18]{wedges}, on $F_\ep(\frk l^L + \frk l^R)$ the conditional law given $Z^\sfi$ of the quantum surface $(D_\ep , h^\wge_{D_\ep} , \eta^\sfi(T_\ep))$ is that of a singly marked quantum disk with boundary length $\nu_{h^\wge}(\bdy D_\ep) = R^\sfi_{T_\ep^-} - R^\sfi_{T_\ep}$.  On the event $F_\ep(\frk l^L + \frk l^R)$, $\nu_{h^\wge}( [0, \eta^\sfi(T_\ep)] ) =  -R^\sfi_{T_\ep} - \frk l^L  $ is determined by $Z^\sfi$.  We thus obtain condition~\ref{item-aux-surface-full}. 
 
Assertion~\ref{item-aux-surface-curve} follows from the locality property of SLE$_6$~\cite[Theorem~6.13]{lawler-book} and the fact that $\eta^\wge$ is independent from $\eta^\sfi$.
\end{proof}

By definition of $F_\ep(\frk l^L + \frk l^R)$, one has $\frk l^L + R^\sfi_{T_\ep^-}  = \frk l^L + o_\ep(1)$, $-R^\sfi_{T_\ep} - \frk l^L   = \frk l^R + o_\ep(1)$, and $\nu_{h^\wge}(\bdy D_\ep\setminus \BB R) \ = o_\ep(1)$ on this event.
Hence Lemma~\ref{lem-cont-peel-conv} is an immediate consequence of Lemma~\ref{lem-aux-surface-law} together with the following lemma, which in turn follows easily from the scaling properties of the quantum disk.

\begin{lem} \label{lem-bdy-cont}
Let $\{(\frk l^L_\ep , \frk l^R_\ep , \delta_\ep) \}_{\ep > 0}$ be a collection of triples of positive real numbers converging to $(\frk l^L , \frk l^R , 0)$ as $\ep\rta 0$. 
For $\ep  >0$, let $\mcl B_\ep = (\BB D , h_\ep , -1,1 )$ be a doubly marked quantum disk with left boundary length $\frk l_\ep^L$ and right boundary length $\frk l_\ep^R$.
Let $x_\ep$ be the point of $\bdy \BB D$ such that the $\nu_{h_\ep}$-length of the counterclockwise arc $[1,x_\ep]_{\bdy\BB D}$ of $\bdy\BB D$ from $1$ to $x_\ep$ is $\delta_\ep$. 

Let $\eta_\ep$ be a chordal SLE$_6$ from $-1$ to $1$ in $\BB D$, sampled independently from $h_\ep$ then parameterized by quantum natural time with respect to $h_\ep$, and let $\wh S_\ep$ be the first time $\eta_\ep$ hits $[1,x_\ep]_{\bdy\BB D}$.  Also let $ Z_{\ep,u} = (L_{\ep,u} , R_{\ep,u} )$ for $u\geq 0$ be the left/right quantum boundary length process for $\wh \eta_\ep$ with respect to $\nu_{h_\ep}$ (which is defined on the same interval as $\eta^\ep$).  For $u\geq 0$, let $  K_{\ep,u}$ be the closure of the points in $\BB D$ which are disconnected from $1$ by $\wh \eta_\ep([0,u])$ and define the quantum surfaces 
\eqb \label{eqn-bdy-cont-surfaces}
\mcl K_{\ep,u} := \left( K_{\ep,u} , h_\ep|_{ K_{\ep,u} } , -1, \eta_\ep(u)   \right) \quad \op{and} \quad
\mcl W_{\ep, u} := \left( \BB D\setminus K_{\ep,u} , h_\ep|_{\BB D\setminus K_{\ep,u} } , \eta_\ep(u), 1  \right)
\eqe 
For any finite collection of times $u_1 ,\dots , u_n \geq 0$, the joint law of $\left(\mcl B_\ep , \eta_{\mcl B_\ep}^\ep|_{[0,\wh S_\ep]} \right)$, $Z_\ep|_{[0,\wh S_\ep]} $, $\{\mcl W_{\ep , u_k \wedge \wh S_\ep  }\}_{k \in [1,n]_{\BB Z} }$, and $\left\{ \left( \mcl K_{\ep , u_k \wedge \wh S_\ep}^\ep , \eta_{\ep, \mcl K_{\ep , u_k \wedge \wh S_\ep}}  \right) \right\}_{k\in [1,n]_{\BB Z}}$ (with each bead of each surface equipped with an extra interior marked point sampled uniformly from its quantum measure) converges as $\ep\rta 0$ to the joint law of $(\mcl B , \eta_{\mcl B}^\nat)$, $Z^\nat|_{[0,S^\nat]}$, $\{\mcl W_{u_k }\}_{k \in [1,n]_{\BB Z} }$ and $\left\{ \left( \mcl K_{u_k }^\nat , \eta^\nat_{ \mcl K_{u_k}^\nat}  \right) \right\}_{k\in [1,n]_{\BB Z}}$ (with each bead of each surface equipped with an extra interior marked point sampled uniformly from its quantum measure) in the topology of curve-decorated quantum surfaces and the Skorokhod topology, as appropriate. 
\end{lem}
\begin{proof}
We will use the scaling property of the quantum disk to construct a coupling of the objects in the statement of the lemma such that the desired convergence conditions hold a.s.
Suppose that our SLE$_6$-decorated quantum disk $(\mcl B , \eta^\nat_{\mcl B})$ is parameterized by $(\BB D , -1, 1)$ instead of $(\BB H , 0,\infty)$, so that $h^\nat$ is a distribution on $\BB D$ and $\eta^\nat$ is an SLE$_6$ in $\BB D$ from $-1$ to $1$. 

Let $\frk l:= \frk l^L + \frk l^R$ and $\frk l_\ep := \frk l^L_\ep + \frk l^R_\ep$ for $\ep >0$.  For $\ep > 0$, let $h_\ep^\nat := h^\nat + \frac{2}{\sqrt{8/3} }  \log(\frk l_\ep / \frk l)$ and note that adding $\frac{2}{\sqrt{8/3} }  \log(\frk l_\ep / \frk l)$ has the effect of multiplying areas by $(\frk l_\ep/\frk l )^{2}$ and boundary lengths by $ \frk l_\ep /\frk l$.  Let $\eta_\ep^\nat(u) := \eta ( (\frk l_\ep/\frk l)^{-3/2} u  )$. By the scaling property of quantum natural time (established in~\cite[Section 6.2]{sphere-constructions}) it follows that $\eta_\ep^\nat $ is parameterized by quantum natural time with respect to $h_\ep^\nat $. Define $Z_{\ep,u}^\nat :=  (\frk l_\ep/\frk l) Z^\nat_{ (\frk l_\ep/\frk l)^{-3/2} u}$, so that $Z_{\ep,u}^\nat$ is the left/right quantum boundary length process for $\eta_\ep^\nat$ with respect to $h_\ep^\nat$. 
 
Let $y_\ep^\nat$ (resp.\ $x_\ep^\nat$) be chosen so that the $\nu_{h_\ep^\nat}$-length of the arc $[-1 , y_\ep]_{\bdy \BB D}$ (resp.\ $[-1,x_\ep]_{\bdy\BB D}$) is $\frk l_\ep^R$ (resp.\ $\frk l_\ep^R + \delta_\ep$). By our choice of $\frk l_\ep$, we have $\nu_{h_\ep^\nat}([y_\ep , -1]_{\bdy\BB D}) = \frk l_\ep^L$.  Also let $\wh S_\ep^\nat$ be the first time at which $\eta_\ep^\nat$ hits $[y_\ep ,x_\ep]_{\bdy\BB D}$. 

Define the quantum surface $\mcl B_\ep^\nat = \left( \BB D , h_\ep^\nat  ,  -1 , y_\ep \right)$.  Then the joint law of the curve-decorated quantum surface $\left(\mcl B_\ep^\nat , \eta_{\ep,\mcl B_\ep^\nat}^\nat |_{[0,\wh S_\ep^\nat]} \right)$ and the process $Z^{ \nat}_\ep |_{[0,\wh S_\ep^\nat]}$ is the same as the joint law of the objects described in the statement of the lemma.
It we define the surfaces as in~\eqref{eqn-bdy-cont-surfaces} (but with $y_\ep$ in place of 1), then since $(\frk l^L_\ep , \frk l^R_\ep , \delta_\ep) \rta (\frk l^L , \frk l^R , 0)$ it is readily verified that the convergence conditions in the statement of the lemma are satisfied.
\end{proof}

\subsubsection{Radon-Nikodym derivative estimate and conclusion of the proof}

In the next lemma, we will use Bayes' rule to determine the conditional law of the quantum surface $\mcl W_{\ep,u}$ (defined in~\eqref{eqn-cond-bubble-surface}) and the curve-decorated quantum surface $(\mcl K_u^\wge , \eta_{\mcl K_u^\wge}^\wge)$ (defined in~\eqref{eqn-wge-hull-surface}) given the event $F_\ep(\frk l^L +\frk l^R)$ and a regularity event which we now define. 
For $\delta>0$, let 
\eqb \label{eqn-reg-event}
E_\delta^\wge(u) := \left\{ \inf_{v\in [0,u]} L^\wge_v \geq -\frk l^L +\delta \right\} \cap \left\{ \inf_{v\in [0,u]} R^\wge_v \geq -\frk l^R +\delta \right\} = \left\{  K_u^\wge \subset [\xi(-\frk l^L+\delta) , \xi(\frk l^R - \delta)] \right\} .
\eqe 
We also define the analog of $E_\delta^\wge(u)$ for the curve-decorated quantum surface $(\mcl B  , \eta^\nat_{\mcl B})$ and the corresponding process $Z^\nat$,  
\eqb \label{eqn-reg-event-bead}
E_\delta^\nat(u) := \left\{ \inf_{v\in [0,u]} L^\nat_v \geq -\frk l^L +\delta \right\} \cap \left\{ \inf_{v\in [0,u]} R^\nat_v \geq -\frk l^R +\delta \right\}  
\eqe 
We observe that with $S^\wge$ and $S^\nat$ the times in the statement of Theorem~\ref{thm-wge-bead-rn}, 
\eqb \label{eqn-trunc-to-time}
\bigcup_{\delta>0} E_\delta^\wge(u) = \{u < S^\wge\}   \quad \op{and} \quad \bigcup_{\delta>0} E_\delta^\nat(u) = \{u < S^\nat\}   ,  \quad \forall u > 0 
\eqe 
where in the second equality we use the formula~\eqref{eqn-S^bead-formula}.
 
\begin{lem} \label{lem-cont-peel-rn}
Let $u\geq 0$ and $\delta >0$ and let $E_\delta^\wge(u)$ and $E_\delta^\nat(u)$ be as above.
For each $\ep \in (0, (\delta/2)^{(1-\zeta)^{-1}})$, the following holds. 
\begin{enumerate} 
\item \label{item-cont-peel-unexplored} The conditional law of the quantum surface $\mcl W_{\ep,u}$ given $Z^\sfi$, $(\mcl K_u^\wge , \eta_{\mcl K_u^\wge}^\wge)$, and the event $E_\delta^\wge(u) \cap F_\ep(\frk l^L + \frk l^R)$ is that of a doubly marked quantum disk with left/right boundary lengths $\frk l^L + L_u^\wge - R^\sfi_{T_\ep^-} $ and $-R^\sfi_{T_\ep} - \frk l^L + R_u^\wge $. 
\item \label{eqn-cont-peel-rn} The conditional law of the curve-decorated quantum surface $(\mcl K_u^\wge , \eta_{\mcl K_u^\wge}^\wge)$ given $ E_\delta^\wge(u)\cap \{u < \wh S\}   \cap F_\ep(\frk l^L + \frk l^R)$ is absolutely continuous with respect to its conditional law given only $E_\delta^\wge(u)$, and the Radon-Nikodym derivative takes the form
\eqb 
(1 +o_\ep(1)) \BB P\left[ E_\delta^\nat(u)   \right]^{-1} \left( \frac{L_u^\wge + R_u^\wge}{\frk l^L + \frk l^R} + 1 \right)^{-5/2}\BB 1_{E_\delta^\wge(u)} 
\eqe 
where here the rate of the~$o_\ep(1)$ depends only on~$\delta$, $\frk l^L$, and~$\frk l^R$. \label{item-cont-peel-explored} 
\end{enumerate}
\end{lem}

For the proof of Lemma~\ref{lem-cont-peel-rn}, we need the following lemma.

\begin{lem}
\label{lem-wge-ind}
For each $u \geq 0$, the past and future curve-decorated quantum surfaces $(\mcl K_u^\wge , \eta^\wge_{\mcl K_u^\wge})$ and $(\mcl W_u^\wge , \eta^\wge_{\mcl W_u^\wge})$ defined in~\eqref{eqn-wge-hull-surface} are independent. 
\end{lem}
\begin{proof}
By~\cite[Theorem~1.18]{wedges}, for each $u\geq 0$ the curve-decorated quantum surface $(\mcl W_u^\wge , \eta^\wge_{\mcl W_u^\wge})$ has the law of a $\sqrt{8/3}$-quantum wedge decorated by an independent chordal SLE$_6$ between its two marked points, parameterized by quantum natural time. By~\cite[Theorem~1.16]{wedges}, $(\mcl W_u^\wge , \eta^\wge_{\mcl W_u^\wge})$ is a.s.\ determined by the quantum surfaces parameterized by the bubbles lying to the left and right of $\eta^\wge_{\mcl W_u^\wge}$, which in turn are a.s.\ determined by $(Z^\wge - Z_u^\wge)|_{[u,\infty)}$ and the singly-marked quantum surfaces parameterized by the bubbles disconnected from $\infty$ by $\eta^\wge $ after time $u$ in the manner described in~\cite[Figure~1.15, Line~3]{wedges}.

By~\cite[Lemma~3.7]{gwynne-miller-char}, the coupling described in Section~\ref{sec-peanosphere}, and local absolute continuity between a $\sqrt{8/3}$-quantum wedge and a doubly marked quantum disk near their respective first marked points, $(\mcl K_u^\wge , \eta_{\mcl K_u^\wge}^\wge)$ is a.s.\ determined by $Z^\wge |_{[0,u]}$ and the singly-marked quantum surfaces parameterized by the bubbles disconnected from $\infty$ by $\eta^\wge |_{[0,u]}$.

Since $Z^\wge$ is a pair of independent totally asymmetric $3/2$-stable processes with no upward jumps, we infer that $Z^\wge|_{[0,u]}$ and $(Z^\wge - Z_u^\wge)|_{[u,\infty)}$ are independent.  Furthermore, by~\cite[Theorem~1.18]{wedges} the singly-marked quantum surfaces parameterized by the bubbles cut out by $\eta^\wge$ are conditionally independent quantum disks given $Z^\wge$, with boundary lengths determined by maginitudes of the downward jumps of $Z^\wge$ at the times the bubbles are disconnected from $\infty$. In particular, the conditional law given $Z^\wge$ of the bubbles in $\mcl K_u^\wge$ (resp.\ $\mcl W_u^\wge$) depends only on $Z^\wge|_{[0,u]}$ (resp.\ $(Z^\wge - Z_u^\wge)|_{[u,\infty)}$).

Combining the three preceding paragraphs shows that $(\mcl K_u^\wge , \eta_{\mcl K_u^\wge}^\wge)$ and $(\mcl W_u^\wge , \eta^\wge_{\mcl W_u^\wge})$ are independent.
\end{proof}

\begin{proof}[Proof of Lemma~\ref{lem-cont-peel-rn}]
Define the unexplored quantum surface $\mcl W_u^\wge$ for $\eta^\wge$ as in~\eqref{eqn-wge-hull-surface}. 
For convenience, fix an embedding $h_u^\wge$ of $\mcl W_u^\wge$ into $(\BB H , 0 ,\infty)$.
Let $\xi_u : \BB R\rta \BB R$ be the increasing function which parameterizes $\BB R$ according to $\nu_{h_u^\wge}$-length with $\xi_u (0) = 0$. 
If $E_\delta^\wge(u)$ occurs, then the point $\xi(-\frk l^L) = \eta^\sfi(0) \in \BB R$ corresponds to a point of $\BB R$ under our given embedding of $\mcl W_u^\wge$, namely the point $  \xi_u (-\frk l^L - L_u^\wge)$.
  
Let $\tau_u$ be the first time $\eta^\sfi$ hits $\eta^\wge([0,u])$ and note that $\{\tau_u = \infty\}  = \{u < \wh S\}$. By the locality property of SLE$_6$~\cite[Theorem~6.13]{lawler-book} and since $\eta^\sfi$ and $\eta^\wge$ are conditionally independent given $h^\wge$, the conditional law of the curve $\eta^\sfi_{\mcl W_u^\wge}|_{[0,\tau_u]}$ (viewed as a curve in $\BB H$, using the embedding $h_u^\wge$) given $(\mcl W_u^\wge , \eta^\wge_{\mcl W_u^\wge})$ and $(\mcl K_u^\wge , \eta_{\mcl K_u^\wge}^\wge)$ is that of a chordal SLE$_6$ from $\xi_u^\wge(-\frk l^L - L_u^\wge)$ to $\infty$ in $\BB H$, stopped at the first time it hits $\xi_u^\wge\left([-L_u^\wge , R_u^\wge] \right)$.  

Let $\eta_u^\sfi$ be a random curve whose conditional law given $(\mcl W_u^\wge , \eta^\wge_{\mcl W_u^\wge})$ and $(\mcl K_u^\wge , \eta_{\mcl K_u^\wge}^\wge)$ is that of a chordal SLE$_6$ from $\xi_u^\wge(-\frk l^L - L_u^\wge)$ to $\infty$ in $\BB H$ parameterized by quantum natural time with respect to $h_u^\wge$, which agrees with the embedding into $\BB H$ of $\eta^\sfi_{\mcl W_u^\wge}  $ until time $\tau_u$ (such a curve exists by the above application of the locality property of SLE$_6$).  
Let $Z_{u,v}^\sfi = (L_{u,v}^\sfi , R_{u,v}^\sfi)$ for $v\geq 0$ be the left/right boundary length process for the pair $(h_u^\wge ,\eta_u^\sfi)$, so that $Z_{u,\cdot}^\sfi \eqD Z^\sfi$. 

Let $F_{u,\ep}(r)$ for $r\geq 0$ be defined as in~\eqref{eqn-cont-peel-event} with $R_{u,\cdot}^\sfi$ in place of $R^\sfi$. 
One easily checks using~\eqref{eqn-cont-peel-property} that for $\ep \in (0, (\delta/2)^{(1-\zeta)^{-1}})$, 
\eqb \label{eqn-cont-peel-event-show0}
E_\delta^\wge(u) \cap \{u < \wh S\} \cap F_\ep(\frk l^L  + \frk l^R)  = E_\delta^\wge(u) \cap  F_{u,\ep}(\frk l^L + \frk l^R + L_u^\wge + R_u^\wge) .
\eqe
On the event $\{u  <\wh S\} \cap F_\ep(\frk l^L +\frk l^R)$, the quantum surface $\mcl W_{\ep,u}$ is the quantum surface parameterized by the bubble disconnected from $\infty$ by $\eta^\sfi_u$ which contains $\xi_u^\wge\left([-L_u^\wge , R_u^\wge] \right)$ on its boundary. By~\cite[Theorem~1.18]{wedges} and Lemma~\ref{lem-wge-ind}, on the event $ E_\delta^\wge(u) \cap  F_{u,\ep}(\frk l^L + \frk l^R + L_u^\wge + R_u^\wge) $ the conditional law of $\mcl W_{\ep,u}$ given the left/right quantum boundary length process $Z_{u,\cdot}^\sfi$ and $(\mcl K_u^\wge , \eta_{\mcl K_u^\wge}^\wge)$ (which together determine $Z^\sfi$ on this event) is that of a doubly marked quantum disk with left/right boundary lengths $\frk l^L  + L_u^\wge + R^\sfi_{T_\ep^-}  $ and $-R^\sfi_{T_\ep} - \frk l^L + R_u^\wge $. By combining this with~\eqref{eqn-cont-peel-event-show0} we see that assertion~\ref{item-cont-peel-unexplored} holds.
 
By Bayes' rule, the Radon-Nikodym derivative of the conditional law of $(\mcl K_u^\wge , \eta_{\mcl K_u^\wge}^\wge)$ given either of the two events in~\eqref{eqn-cont-peel-event-show0} with respect to its conditional law given only $E_\delta^\wge(u)$ is equal to
\eqb \label{eqn-cont-peel-bayes}
\frac{\BB P\left[  F_{u,\ep}(\frk l^L + \frk l^R + L_u^\wge + R_u^\wge) \,|\,   (\mcl K_u^\wge , \eta_{\mcl K_u^\wge}^\wge) \right] }{ \BB P\left[ E_\delta^\wge(u) \cap \{u < \wh S\} \cap F_\ep(\frk l^L  + \frk l^R) \right]  } \BB 1_{E_\delta^\wge(u)} ; 
\eqe 
here we note that $E_\delta^\wge(u) \in \sigma(\mcl K_u^\wge , \eta_{\mcl K_u^\wge}^\wge)$.
By Lemma~\ref{lem-cont-peel-prob}, on $E_\delta^\wge(u)$ it holds that
\eqb \label{eqn-cont-peel-numerator}
\BB P\left[  F_{u,\ep}(\frk l^L + \frk l^R + L_u^\wge + R_u^\wge) \,|\,   (\mcl K_u^\wge , \eta_{\mcl K_u^\wge}^\wge) \right] = \left(\frac32 +o_\ep(1) \right) \ep^{5/2}  \left(  \frk l^L + \frk l^R + L_u^\wge + R_u^\wge \right)^{-5/2}   .
\eqe  
Furthermore,  
\allb \label{eqn-cont-peel-denominator}
 \BB P\left[ E_\delta^\wge(u) \cap \{u < \wh S\} \cap F_\ep(\frk l^L  + \frk l^R) \right] 
&= \BB P\left[ E_\delta^\wge(u) \cap \{u < \wh S\} \,|\,  F_\ep(\frk l^L  + \frk l^R) \right] \BB P\left[F_\ep(\frk l^L  + \frk l^R) \right] \notag  \\
&=     \BB P\left[ E_\delta^\wge(u) \cap \{u < \wh S\} \,|\,  F_\ep(\frk l^L  + \frk l^R) \right]  \left(\frac32 +o_\ep(1) \right) \ep^{5/2} \left(\frk l^L + \frk l^R \right)^{-5/2}    .
\alle
By Lemma~\ref{lem-cont-peel-conv}, the conditional law of $Z^\wge|_{[0, \wh S ]} $ given $F_\ep(\frk l^L  + \frk l^R)$ converges to the law of $ Z^\nat|_{[0,S^\nat]}$ with respect to the Skorokhod topology so 
\eqb \label{eqn-cont-peel-cond-term}
\BB P\left[ E_\delta^\wge(u) \cap \{u < \wh S\} \,|\,  F_\ep(\frk l^L  + \frk l^R) \right] = \BB P\left[ E_\delta^\nat(u)   \right] (1+o_\ep(1)) 
\eqe 
where here we recall that $E_\delta^\nat(u) \subset \{u  <S^\nat\}$. In all of the above formulas, the rate of the $o_\ep(1)$ depends only on $\delta$, $\frk l^L$, and $\frk l^R$. Plugging~\eqref{eqn-cont-peel-numerator},~\eqref{eqn-cont-peel-denominator}, and~\eqref{eqn-cont-peel-cond-term} into~\eqref{eqn-cont-peel-bayes} yields assertion~\ref{item-cont-peel-explored}.  
\end{proof}

\begin{proof}[Proof of Theorem~\ref{thm-wge-bead-rn}]
 By Lemmas~\ref{lem-cont-peel-conv} and~\ref{lem-cont-peel-rn}, we can send $\ep\rta 0$ to obtain that for each $u\geq 0$ and each $\delta>0$, the following is true. 
\begin{enumerate}
\item The conditional law of the quantum surface $\mcl W_u^\nat$ given the curve-decorated quantum surface $(\mcl K_u^\nat , \eta^\nat_{\mcl K_u^\nat})$ and the event $E_\delta^\nat(u)  $ in~\eqref{eqn-reg-event-bead} is that of a doubly marked quantum disk with left/right boundary lengths $\frk l^L   + L_u^\nat$ and $\frk l^R + R_u^\nat $. \label{item-wge-bead-rn-unexplored}
\item The conditional law of $(\mcl K_u^\nat , \eta^\nat_{\mcl K_u^\nat})$ given $E_\delta^\nat(u)  $ is absolutely continuous with respect to the conditional law of $(\mcl K_u^\wge , \eta^\nat_{\mcl K_u^\wge})$ given the event $E_\delta^\wge(u)$ of~\eqref{eqn-reg-event}, with Radon-Nikodym derivative  
\eqb \label{eqn-rn-deriv-truncated}
\BB P\left[ E_\delta^\nat(u) \right]^{-1}  \left( \frac{L_u^\wge + R_u^\wge}{\frk l^L + \frk l^R} + 1 \right)^{-5/2} \BB 1_{E_\delta^\wge(u)}
\eqe 
\end{enumerate}
Note that to obtain statement~\ref{item-wge-bead-rn-unexplored}, we use that the conditional law of a doubly marked quantum disk with given left/right boundary lengths depends continuously on these boundary lengths in the topology of Section~\ref{sec-surface-curve}, which follows from a scaling argument as in Lemma~\ref{lem-bdy-cont}.

The second assertion of the theorem statement is precisely statement~\ref{item-wge-bead-rn-unexplored} above. 
Sending $\delta \rta 0$ in~\eqref{eqn-rn-deriv-truncated} and recalling~\eqref{eqn-trunc-to-time} shows that the conditional law of $(\mcl K_u^\nat , \eta^\nat_{\mcl K_u^\nat})$ given $\{u < S^\nat\}  $ is absolutely continuous with respect to the conditional law of $(\mcl K_u^\wge , \eta^\nat_{\mcl K_u^\wge})$ given $\{u < S^\wge\}$, with Radon-Nikodym derivative  
\eqbn
\BB P\left[ u < S^\nat \right]^{-1}  \left( \frac{L_u^\wge + R_u^\wge}{\frk l^L + \frk l^R} + 1 \right)^{-5/2} \BB 1_{(u < S^\wge)} .
\eqen
Multiplying this Radon-Nikodym derivative estimate by $\BB P[u < S^\nat]$ yields~\eqref{eqn-wge-bead-rn}. 
\end{proof}

\appendix

\section{Quantum disks}
\label{sec-disk}

To make this work more self-contained, in this appendix we recall the precise definition of the quantum disk from~\cite{wedges}. We will give the direct definition of this quantum surface in terms of Bessel excursions. There are also other equivalent ways of defining quantum disks, e.g., as limits of certain surfaces constructed from a free-boundary GFF~\cite[Appendix A]{wedges} or as the quantum surfaces parameterized by the bubbles cut out by an SLE$_{\kappa'}$ curve for $\kappa'  =16/\gamma^2$ on a certain LQG surface~\cite[Theorems 1.16 and 1.17]{wedges}.  We also expect that the definition of the quantum disk given here is equivalent to the one in~\cite{hrv-disk}. It is likely that this could be proved using similar techniques to the ones in~\cite{ahs-sphere} (which proves the analogous statement for quantum spheres).

We first define an infinite measure on doubly marked quantum disks, then condition this measure on certain events to obtain probability measures, then forget one marked point to obtain singly marked quantum disks. It is convenient to define a quantum disk parameterized by the infinite strip $\mcl S = \BB R\times (0,\pi)$ since the field takes a simpler form in this case (one can parameterize by the unit disk instead by applying a conformal map and using~\eqref{eqn-lqg-coord}). Let $\mcl H^0(\mcl S)$ (resp.\ $\mcl H^\dagger(\mcl S)$) be the Hilbert space of mean-zero functions on $\mcl S$ with finite Dirichlet energy which are constant (resp.\ have mean zero) on each horizontal line segment $\{x\} \times [0,\pi]$. Then the space of all mean-zero functions on $\mcl S$ with finite Dirichlet energy is the orthogonal direct sum of $\mcl H^0(\mcl S)$ and $\mcl H^\dagger(\mcl S)$~\cite[Lemma 4.3]{wedges}. The following definition is given in~\cite[Section~4.5]{wedges}. 

\begin{defn} \label{def-infinite-disk}
For $\gamma \in (0,2)$, the \emph{infinite measure on quantum disks} is the measure $ \mcl M^{\op{disk}}$ on doubly marked quantum surfaces $(\mcl S ,h , -\infty,\infty)$ defined as follows. 
\begin{itemize}
\item ``Sample" $e$ from the infinite excursion measure of a Bessel process of dimension $3-\frac{4}{\gamma^2}$ (see~\cite[Remark 3.7]{wedges}). Let $h^0 \in \mcl H^0(\mcl S)$ be the function whose common value on each segment $\{x\}\times [0,\pi]$ is given by the process $2\gamma^{-1} \log e$ reparameterized to have quadratic variation $2\,dx$. Note that this process is only defined modulo translations -- different translations give equivalent quantum surfaces.
\item Let $h^\dagger$ be sampled from the law of the projection of a free-boundary GFF on $\mcl S$ onto $\mcl H^\dagger(\mcl S)$, independent from $h^0$.
\item Let $h = h^0 + h^\dagger$. 
\end{itemize}
\end{defn}

We note that for each $t> 0$, the measure $\mcl M^{\op{disk}}$ assigns finite mass to the set of surfaces whose corresponding excursion $e$ has time length at least $t$ (as the Bessel excursion measure assigns finite mass to those Bessel excursions with length at least $t$). This implies that for each $\frk a , \frk l > 0$, $\mcl M^{\op{disk}}$ assigns finite mass to the set of surfaces with LQG area at least $\frk a$ and/or LQG boundary length at least $\frk l$. 

The space $\BB M_k^{\op{CPU}}$ defined in Section~\ref{sec-surface-curve} is a separable metric space, and we can view quantum surfaces with boundary as elements of this space equipped with the curve which parameterizes the boundary according to the LQG boundary length measure. Hence it makes sense to talk about regular conditional laws for quantum surfaces. The following corresponds to~\cite[Definition 4.21]{wedges}.

\begin{defn} \label{def-disk} 
Let $\mcl M^{\op{disk}}$ be as in Definition~\ref{def-infinite-disk}. 
\begin{itemize}
\item For $\frk l > 0$, the \emph{quantum disk with boundary length $\frk l$} is the regular conditional distribution of the measure $  d\mcl M^{\op{disk}}(h)$ given that $\nu_h(\bdy\mcl S) = \frk l$.
\item For $\frk l^L ,\frk l^R > 0$, the \emph{quantum disk with left/right boundary lengths $\frk l^L, \frk l^R > 0$} is the regular conditional distribution of the measure $ d\mcl M^{\op{disk}}(h)$ given that $\nu_h(\BB R\times \{\pi\}) = \frk l^L$ and $\nu_h(\BB R\times \{0\}) = \frk l^R$. 
\end{itemize}
\end{defn}

The above definitions give us doubly marked quantum disks. One defines singly marked quantum disks by forgetting one of the marked points for a doubly marked quantum disk (i.e., projecting from $\BB M_2^{\op{CPU}}$ to $\BB M_1^{\op{CPU}}$). 
It is shown in~\cite[Proposition~A.8]{wedges} that the marked points for a doubly marked quantum disk are independent samples from the $\gamma$-LQG boundary length measure if we condition on the disk viewed as an unmarked quantum surface. Equivalently, suppose that $(\mcl S, h, -\infty,+\infty)$ is a quantum disk, that $x,y \in \partial \mcl S$ are picked independently from the $\gamma$-LQG boundary measure $\nu_h$, and that $\varphi \colon \mcl S \to \mcl S$ is a conformal transformation with $\varphi(-\infty) = x$ and $\varphi(+\infty) = y$.  Then with $\wt{h} = h \circ \varphi + Q \log|\varphi'|$, we have that~$\wt{h}$ and~$h$ have the same law modulo a horizontal translation of~$\mcl S$.

A priori the regular conditional laws of the infinite quantum disk measure given the boundary length only make sense for a.e.\ $\frk l > 0$. However, as explained just after~\cite[Definition~4.21]{wedges}, the quantum disk possesses a scale invariance property which allows us to define these regular conditional laws for every choice of $\frk l > 0$. In particular, if $(\mcl S  , h , -\infty,\infty)$ is a quantum disk with boundary length $\frk l$ and $C >0$, then $(\mcl S , h + C , -\infty,\infty)$ is a quantum disk with boundary length $e^{\gamma C} \frk l$. To make sense of quantum disks with arbitrary choices of left/right boundary lengths $\frk l^L , \frk l^R> 0$, one can, for example, proceed as follows. Since the marked points for a doubly marked quantum disk are independent samples from the LQG boundary length measure, if we start with a singly marked quantum disk of boundary length $\frk l^L + \frk l^R$ and choose a second (random) marked point by the condition that the boundary lengths of the arcs separating the two marked points have quantum lengths $\frk l^L$ and $\frk l^R$, respectively, then we obtain a doubly marked quantum disk with left/right boundary lengths $\frk l^L$ and $\frk l^R$.

\bibliography{cibiblong,cibib,perc-ref}
\bibliographystyle{hmralphaabbrv}

\end{document}